\documentclass[]{amsart}

\usepackage[Symbol]{upgreek}

\usepackage{graphicx}
\usepackage{enotez}
\usepackage{amssymb}
\usepackage{epic}
\usepackage{mathrsfs}
\usepackage{accents}
\usepackage{array}
\usepackage{stmaryrd}
\usepackage{enumitem}
\usepackage{longtable}
\usepackage{hyperref} 

\usepackage{etoolbox} 

\usepackage{tikz}

\input{xy}
\xyoption{matrix}
\xyoption{arrow}

\numberwithin{equation}{section}

\makeatletter
\renewcommand{\tocsection}[3]
 { \indentlabel{\@ifnotempty{#2}{\parbox{2.5em}{\ignorespaces#1 #2.}\quad}}#3}

\makeindex

\newcommand{\bbold}{\mathbb}

\def\R { {\bbold R} }
\def\Q { {\bbold Q} }
\def\Z { {\bbold Z} }
\def\C { {\bbold C} }
\def\N { {\bbold N} }
\def\T { {\bbold T} }

\def\c {\mathcal{C}}

\def \I{\operatorname{I}}

\def \Ex{\operatorname{E}}
\def \Dx{\operatorname{D}}

\def \order{\operatorname{order}}

\def \ex{\operatorname{e}}
\def \wr {\operatorname{wr}}

\def \Univ {{\operatorname{U}}}

\renewcommand\epsilon{\varepsilon}

\def \d{\operatorname{d}}

\def \<{\langle}
\def \>{\rangle}

\def \tilde {\widetilde}

\def \((  {(\!(}
\def \)) {)\!)}

\def \k {{{\boldsymbol{k}}}}

\DeclareMathSymbol{\precequ}{\mathrel}{symbols}{"16}
\DeclareMathSymbol{\succequ}{\mathrel}{symbols}{"17}

\def \nasymp{\not\asymp}

\renewcommand{\Re}{\operatorname{Re}}
\renewcommand{\Im}{\operatorname{Im}}

\newcommand{\claim}[2][\!\!]{\medskip\noindent {\bf Claim #1:} {\it #2}\medskip}

\newtheorem{theorem}{Theorem}[section]

\newtheorem{lemma}[theorem]{Lemma}
\newtheorem{prop}[theorem]{Proposition}
\newtheorem{cor}[theorem]{Corollary}

\newtheorem*{theoremUnnumbered}{Theorem}

\newtheorem*{theoremA}{Theorem A}
\newtheorem*{theoremB}{Theorem B}

\theoremstyle{definition}

\theoremstyle{remark}
\newtheorem*{example}{Example}

\newtheorem*{remarks}{Remarks}

\newtheorem*{question}{Question}

\newcommand{\abs}[1]{\lvert#1\rvert}

\def \fM {{\mathfrak M}}

\def \No{\text{{\bf No}}}

\def \Zero{\operatorname{Z}}

\let\oldi\i
\let\oldj\j
\renewcommand\i{\relax\ifmmode{\boldsymbol{i}}\else\oldi\fi}
\renewcommand\j{\relax\ifmmode{\boldsymbol{j}}\else\oldj\fi}

\renewcommand\leq{\leqslant}
\renewcommand\geq{\geqslant}
\renewcommand\preceq{\preccurlyeq}

\renewcommand\le{\leq}
\renewcommand\ge{\geq}

\DeclareMathAlphabet{\mathbf}{OML}{cmm}{b}{it}

\DeclareFontFamily{U}{fsy}{}
\DeclareFontShape{U}{fsy}{m}{n}{<->s*[.9]psyr}{}
\DeclareSymbolFont{der@m}{U}{fsy}{m}{n}
\DeclareMathSymbol{\der}{\mathord}{der@m}{182}

\DeclareSymbolFont{der@m}{U}{fsy}{m}{n}
\DeclareMathSymbol{\derdelta}{\mathord}{der@m}{100}

\DeclareSymbolFont{imag@m}{OT1}{cmr}{m}{ui}
\DeclareMathSymbol{\imag}{\mathord}{imag@m}{105}

\DeclareFontFamily{OMS}{smallo}{}
\DeclareFontShape{OMS}{smallo}{m}{n}{<->s*[.65]cmsy10}{}
\DeclareSymbolFont{smallo@m}{OMS}{smallo}{m}{n}
\DeclareMathSymbol{\smallo}{\mathord}{smallo@m}{79}

\DeclareFontFamily{OMS}{largerdot}{}
\DeclareFontShape{OMS}{largerdot}{m}{n}{<->s*[.8]cmsy10}{}
\DeclareSymbolFont{largerdot@m}{OMS}{largerdot}{m}{n}
\DeclareMathSymbol{\largerdot}{\mathord}{largerdot@m}{15}

\DeclareMathSymbol{\llambda}{\mathord}{der@m}{108}
\DeclareMathSymbol{\rrho}{\mathord}{der@m}{114}

\def \Upg{\Upgamma}
\def \upl{\uplambda}
\def \Upl{\Uplambda}
\def \upo{\upomega}
\def \Upo{\Upomega}

\def\HLO{\Upl\Upo}

\newcommand{\equationqed}[1]{\[\pushQED{\qed}#1 \qedhere\popQED\]\let\qed\relax}
\newcommand{\alignqed}[1]{\begin{align*}\pushQED{\qed} #1 \qedhere\popQED\end{align*}\let\qed\relax}

\makeatletter
\newcommand{\dminus}{\mathbin{\text{\@dminus}}}

\newcommand{\@dminus}{%
  \ooalign{\hidewidth\raise1ex\hbox{\bf.}\hidewidth\cr$\m@th-$\cr}%
}
\makeatother

\def \Gr{\mathcal{C}^r}
\def \Gn{\mathcal{C}^n}

\def \Gi{\mathcal{C}^{<\infty}}
\def \Ginf{\mathcal{C}^{\infty}}
\def \Gom{\mathcal{C}^{\omega}}

\def \Calinf{\mathcal{C}^{<\infty}}

\makeatletter
\renewcommand\part{\@startsection{part}{0}%
  \z@{\linespacing\@plus\linespacing}{.5\linespacing}%
  {\normalfont\bfseries\centering}}
\makeatother

\makeatletter
\renewcommand\theindex{\@restonecoltrue\if@twocolumn\@restonecolfalse\fi
  \columnseprule\z@ \columnsep 35\p@
  \twocolumn[\@xp\part\@xp*\@xp{\bf Index}\bigskip]%
  \let\item\@idxitem
  \parindent\z@  \parskip\z@\@plus.3\p@\relax
  \small}

\makeatletter
\newcommand{\smallbullet}{} 
\DeclareRobustCommand\smallbullet{%
  \mathord{\mathpalette\smallbullet@{0.6}}%
}
\newcommand{\smallbullet@}[2]{%
  \vcenter{\hbox{\scalebox{#2}{$\m@th#1\bullet$}}}%
}
\makeatother

\begin{document}

\title{Relative Differential Closure in Hardy Fields}
\author[Aschenbrenner]{Matthias Aschenbrenner}
\address{Kurt G\"odel Research Center for Mathematical Logic\\
Universit\"at Wien\\
1090 Wien\\ Austria}
\email{matthias.aschenbrenner@univie.ac.at}

\author[van den Dries]{Lou van den Dries}
\address{Department of Mathematics\\
University of Illinois at Urbana-Cham\-paign\\
Urbana, IL 61801\\
U.S.A.}
\email{vddries@math.uiuc.edu}

\author[van der Hoeven]{Joris van der Hoeven}
\address{CNRS, LIX (UMR 7161)\\ 
Campus de l'\'Ecole Polytechnique\\  91120 Palaiseau \\ France}
\email{vdhoeven@lix.polytechnique.fr}

\date{January 2025}

\begin{abstract} 
We study relative differential closure in the context of Hardy fields.
Using our earlier work on algebraic differential
equations over Hardy fields,
this leads to a proof of a conjecture of Boshernitzan (1981): the intersection of all maximal analytic Hardy fields
agrees with that of all maximal Hardy fields. 
We also generalize a key ingredient in the proof, and describe a cautionary example
delineating the boundaries of its applicability. 
\end{abstract}

\maketitle

\tableofcontents

\section*{Introduction}

\noindent
Hardy's monograph \textit{Orders of Infinity}\/\endnote{The first edition (1910) of \cite{Ha}   was reviewed for this journal by G\"odel's PhD advisor Hahn~\cite{Hahn}.}\cite{Ha} founded an
asymptotic calculus of non-oscillating real-valued functions, building on
earlier ideas by du Bois-Reymond \cite{dBR75}. Hardy introduced the class of 
{\it logarithmic-exponential functions} ({\it LE-functions}\/, for short):
functions constructed in finitely many steps from real constants and the identity function $x$ using arithmetic operations, exponentiation, and logarithm.
He observed that this class allows one to describe the growth rates at infinity of many functions that
naturally arise in mathematics. In his own words~\cite[p.~48]{Ha}:
\begin{quote}
{\it No function has yet presented itself in analysis the
laws of whose increase, in so far they can be stated at all, cannot be stated, so to
say, in logarithmico-exponential terms.}
\end{quote}
Typical examples of LE-functions defined on $(1, +\infty)$ are $x^r$ ($r \in \R$), $x^x$, $\ex^{x^2}$, 
and~$(\log x)(\log \log x)$.
Much of the usefulness of the class of LE-functions stems from the fact that their germs at
infinity form what Bourbaki~\cite{Bou}  called a {\it Hardy field}\/: a field $H$ of germs at $+\infty$ of differentiable real-valued functions on intervals~$(a,+\infty)$~($a\in\R$)
such that for any such function with germ in~$H$, the germ of its derivative is
also in~$H$. The basic facts about Hardy fields are due to
 Bourbaki~\cite{Bou}, Sj{\"o}din~\cite{Sj}, Robinson~\cite{Ro},  Boshernitzan~\cite{Boshernitzan81}--\cite{Boshernitzan87}, and Rosenlicht~\cite{Ros}--\cite{Rosenlicht95}. As background we shall mention such facts in this introduction. 

A Hardy field~$H$ such that each~$f\in H$ has a smooth ($\Ginf$) representative  is called  {\it smooth};  likewise we define when $H$ is {\it  analytic}\/.
Most Hardy fields from practice (like Hardy's field of LE-functions) are analytic; but not every Hardy field is analytic, or even smooth\endnote{Boshernitzan~\cite{Boshernitzan81} (with details in \cite{Gokhman}) first suggested an example of a non-smooth Hardy field. Rolin, Speissegger, Wilkie~\cite{RSW} construct o-minimal expansions $\tilde{\R}$ of the ordered field of real numbers such that 
  the Hardy field $H$ consisting of the germs of functions~$\R \to\R$ that are definable in $\tilde{\R}$ is smooth but not analytic. Le Gal and Rolin~\cite{LeGalRolin} construct such expansions such that   the corresponding Hardy field $H$ is not smooth.}.
Every Hardy field is naturally
a differential field, and an ordered field: the germ of a function~$f$ is declared to be positive whenever~$f (t)$ is
eventually positive. In addition to the ordering we also use the asymptotic relations~$\preceq$,~$\prec$,~$\asymp$ to compare germs $f$, $g$ in a Hardy field: 
 \begin{align*}
f \preceq g & \quad:\Longleftrightarrow \quad f =O(g)\quad  :\Longleftrightarrow\quad  \text{$| f | \le c |g|$ for some real $c>0$} \\
f \prec g & \quad:\Longleftrightarrow \quad  f =o(g) \quad \ :\Longleftrightarrow\quad  \text{$| f |<c |g|$ for all real $c>0$}\\
 f\asymp g & \quad:\Longleftrightarrow\quad  f\preceq g \text{ and }g\preceq f;\quad \, f\succ g\quad :\Longleftrightarrow\ \quad g\prec f.
\end{align*}
(The $\preceq$-notation of du Bois-Reymond predates the big $O$-notation of Bachmann and Landau, and is
more  convenient in dealing with Hardy fields.)
The basic operations of calculus play well with the ordering and asymptotic relations on  Hardy fields. For example, given any germs $f$, $g$ in a common Hardy field,  
$$f>0,\ f\succ 1\Rightarrow f'>0,\qquad  f\prec g\nasymp 1\Rightarrow f'\prec g'.$$
The germ of a non-oscillating differentially algebraic function usually lies in a Hardy field. Besides the LE-functions, this is also the case for special
functions like the error function $\operatorname{erf}$, the exponential integral $\operatorname{Ei}$, the Airy functions $\operatorname{Ai}$
and~$\operatorname{Bi}$, etc. Many differentially transcendental functions,
like the Riemann $\zeta$-function and Euler's $\Gamma$-function, also have their germs in Hardy fields.   

Characteristic of Hardy fields is that their elements are
non-oscillating in a strong sense: if $f$ is a germ in a Hardy field~$H$, then not only is the sign of $f(t)$   ultimately constant, but
also each differential-polynomial expression in $f$ such as 
$$g(t)f''(t)^3-h(t)f'(t)f(t)^2+2\qquad (g,h\in H).$$ This property is reflected in the existence of a field ordering
on $H$ as well as  the relations~$\preceq$ and~$\prec$ that are so useful in
asymptotics. 
Functions that are non-oscillating in such a strong sense may be viewed as {\it tame}\/. In certain
applications, establishing tameness   is   decisive: for example, it plays an important role in \'Ecalle's work \cite{Ecalle} on Dulac's conjecture
(a weakened version of Hilbert's 16th Problem). An even stronger form of tameness is o-minimality~\cite{Miller}, and the
germs of definable univariate functions of an o-minimal expansion of the real field   form
a Hardy field. This leads to many further examples of Hardy fields~\cite{KRS,RSW}. 

The class of LE-functions is rather small. For example, the antiderivatives of~$\ex^{x^2}$  have their germs in
 a Hardy field but are not LE-functions (Liouville,
 cf.~\cite{Rosenlicht72}). Hardy's quote   above notwithstanding, the functional inverse of~$(\log x)( \log \log x)$
turned out to not even be asymptotic to an LE-function~\cite{DMM1,vdH:PhD} (and yet its germ also lies in a Hardy field).
There are also (analytic) solutions of simple functional equations in Hardy fields ultimately outgrowing all LE-functions~\cite{Boshernitzan86}.

Any Hardy field has a unique algebraic Hardy field extension that is real closed; see ~\cite{Sj} and~\cite{Ro}. For any germ $f$ in a Hardy 
field $H$, its exponential $\ex^f$, its logarithm~$\log f$ (if $f >0$), and any primitive of $f$  
lie in a Hardy field extending~$H$; see~\cite{Bou}.  More generally, if~$H$ is a Hardy field and 
$P,Q\in H[Y]\setminus \{0\}$, then each germ of a $\mathcal C^1$-function $y$ satisfying~$y'Q(y)=P(y)$ belongs to a
Hardy field extending $H$; see~\cite[Theorem~2]{Ros} and~\cite{Singer}. 
In~\cite{ADH2} we proved what is in some sense the
ultimate result on solving algebraic differential equations in Hardy fields:

{\samepage
\begin{theoremUnnumbered}
Given any Hardy field $H$, polynomial $P\in H[Y_0,\dots,Y_n]$ and~$f,g \in H$ with $f<g$ and $P(f,f',\dots,f^{(n)})<0<P(g,g',\dots,g^{(n)})$,   there is a $y$ in a Hardy field extension of $H$
 such that $f<y<g$ and $P(y,y',\dots,y^{(n)})=0$.
\end{theoremUnnumbered}}

\noindent
We convert this into an intermediate value property by means of the concept of
a {\it maximal}\/ Hardy field: a Hardy field that is not contained in any strictly
larger one. Likewise we define maximal smooth Hardy fields and maximal analytic Hardy fields.
Any Hardy field of either type is contained in a maximal one of the same type, by Zorn.
Now the theorem can be rephrased as: maximal Hardy
fields have DIVP (the {\it Differential Intermediate Value Property}\/), where a Hardy field $H$ is said to have DIVP if
for all $P\in H[Y_0,\dots,Y_n]$ and~$f,g \in H$ such that~$f<g$ and~${P(f,f',\dots,f^{(n)})\, <\, 0\, <\, P(g,g',\dots,g^{(n)})}$,
there is a $y\in H$ such that~${f<y<g}$ and $P(y,y',\dots,y^{(n)})=0$.  By the way, \cite{ADH2} also shows that maximal smooth and maximal analytic Hardy fields have DIVP.

DIVP  essentially captures all properties of maximal Hardy fields that
can be stated in the language of ordered differential fields, in
analogy with the intermediate value property for ordinary polynomials
capturing the property for an ordered field to be real closed. For a
further explanation of this statement and numerous consequences of the
above theorem\endnote{G\"odel's Completeness Theorem from his PhD thesis~\cite{Goedel30} is in the background of the model theoretic tools used in proving the completeness of the theory of maximal Hardy fields, which in particular entails the existence of a decision procedure for this theory.} we refer to the introduction to \cite{ADH2}.

Some (germs of) functions are ``absolutely tame'' in the sense that they belong to {\it every}\/
maximal Hardy field. This holds in particular for all LE-functions, 
and Boshernitzan~\cite{Boshernitzan81, Boshernitzan82,Boshernitzan86} promoted the study of these germs as   natural generalizations of
 LE-functions. 
The Hardy field~$\Ex$ of absolutely tame germs is rather extensive: for example,
it is closed under exponentiation, logarithm, and taking antiderivatives; more generally,  any solution $y\in\mathcal C^1$ of an equation~$y'Q(y)=P(y)$
where~${P,Q\in \Ex[Y]\setminus\{0\}}$ also has its germ in $\Ex$ (by a result mentioned earlier). Thus, for instance, $\arctan$ and the Gaussian integrals $\int^x \ex^{-t^2}dt$  are in $\Ex$. Moreover, if~$f\in\Ex$ and~${f\preceq 1}$, then~$\cos f,\sin f\in\Ex$. But Boshernitzan~\cite[Proposition~3.7]{Boshernitzan86} also exhibited germs in Hardy fields not belonging to~$\Ex$:
the germ of {\it any}\/ $\c^2$-function satisfying~$y''+y=\ex^{x^2}$
lies in a Hardy field\endnote{The existence of a solution to this
equation in each maximal Hardy field $H$ follows from the above theorem
with $P(Y)=Y''+Y-\ex^{x^2}$, $f =0$, and $g=\ex^{x^2}$.},
but   no maximal Hardy field   contains more than one\endnote{The difference of two distinct solutions to the equation $y''+y=\ex^{x^2}$ is a nonzero solution to the
homogeneous equation $y+y''=0$ and thus oscillating.} and hence none of them lies in~$\Ex$. As a consequence, $\Ex$ does not have DIVP\endnote{This argument  also suggests that there exist many maximal Hardy fields: in \cite{ADH6}, we show that there are actually $2^{2^{\aleph_0}}$
many.}. Boshernitzan also considered the intersection $\Ex^\infty$ of all maximal smooth Hardy fields and 
the intersection $\Ex^\omega$ of all maximal analytic Hardy fields, and conjectured
that $\Ex = \Ex^\infty = \Ex^\omega$ \cite[\S{}10, Conjecture~1]{Boshernitzan81}.
As positive evidence, in \cite[(20.1)]{Boshernitzan82} he obtained $\Ex\subseteq\Ex^\infty\subseteq\Ex^\omega$. In this paper we prove this conjecture:
  
\begin{theoremA}
Any germ lying in all maximal analytic Hardy fields also lies in all maximal Hardy fields.
\end{theoremA}

\noindent
By \cite[Theorem~14.3]{Boshernitzan82}, each~$f\in\Ex^\omega$ is differentially algebraic. So it may not be surprising that the proof of Theorem~A centers on a study of the notion of relative differential closure in the context
of Hardy fields: if $E\supseteq H$ are Hardy fields, the {\it differential closure}\/ of~$H$ in~$E$ is the set
of all $y\in E$ which are differentially algebraic over $H$, that is, satisfy an equation~$P(y,y',\dots,y^{(n)})=0$
for some $n$ and nonzero~$P\in H[Y_0,\dots,Y_n]$, and we say that $H$ is {\it differentially closed}\/ in $E$ if 
it equals its differential closure   in~$E$.
 A crucial ingredient for this study is a   differential transcendence result~[ADH, 16.0.3], of which we
prove here a variant:

\begin{theoremB}
Let $E\supseteq H$ be Hardy fields where $H$ properly extends $\R$ and has \textup{DIVP}.
Then $H$ is differentially closed in $E$ iff $E\cap \exp(H) \subseteq H$.
\end{theoremB}

\noindent
We   include an example showing that the hypothesis $H\supset\R$ in Theorem~B cannot be replaced by the condition
$E\cap\R=H\cap\R\ne H$.   This uses a result of
Rosenlicht~\cite{Rosenlicht74}.

\subsection*{From du Bois-Reymond and Hardy to \'Ecalle, Conway, and G\"odel.}
An attractive feature of Hardy fields is that their elements are actual functions (more precisely, germs of such). 
To conclude this introduction we recall two alternative universal frameworks for tame asymptotics in which growth rates are explicitly represented, namely {\it transseries}\/ and {\it surreal numbers.}\/
Trans\-series (\`a la \'Ecalle~\cite{Ecalle}) are constructed from a formal indeterminate $x$ and the real numbers using the field operations, exponentiation, logarithms, and certain kinds of {\it infinite}\/ summation\endnote{For a construction and basic properties of the field of transseries, see [ADH, Appendix~A]. A more leisurely exposition is in~\cite{Edgar}.}. They are generalized series (as in Hahn~\cite{Hahn07}) with real coefficients and monomials that themselves
are exponentials of ``simpler'' transseries, such as the third term of 
$$f = \ex^{\frac{1}{2}\ex^x}-5\ex^{x^2}+\ex^{x+x^{1/2}+x^{1/3}+\cdots}+\sqrt[3]{2}\log x-x^{-1}+\ex^{-x}+\ex^{-2x}+\cdots+5\ex^{-x^{3/2}}.$$
The transseries form a field which,
like each Hardy field, naturally comes equipped with a derivation $\frac{d}{d x}$ and an ordering making it an ordered differential field~$\T$.

Surreal numbers, invented by J.~H.~Conway~\cite{ONAG} in connection with game theory, have a more combinatorial flavor. 
They include both real numbers and Cantor's ordinal numbers, forming a proper class naturally equipped with an ordering and arithmetic operations making it a real closed ordered field extension $\No$ of $\R$. For example, with $\omega$ the first infinite ordinal,~$\omega-\pi$,  $1/\omega$, $\sqrt{\omega}$,
make sense as surreal numbers\endnote{A brief summary of work on $\No$ past Conway is in \cite{Ryba}.
We also refer to \cite{Bagayoko:PhD} for an interpretation of surreal
numbers in terms of growth rates at infinity.}. Recently,
Berarducci and Mantova~\cite{BM} constructed a derivation~$\der_{\operatorname{BM}}$ on~$\No$
making it an ordered differential field with field of constants $\R$ and~${\der_{\operatorname{BM}}(\omega)=1}$.

Thus maximal Hardy fields, the field of transseries $\T$, and $\No$ are all extensions of~$\R$ to
ordered differential fields containing both infinite and infinitesimal  elements. 
It is natural to ask about the canonicity of such extensions of the continuum. By~[ADH] and \cite{ADH1+,ADH2}, they share the same first-order theory, as ordered differential fields.
In~\cite{ADH6}  we show that under Cantor's Continuum Hypothesis (CH), any maximal Hardy field is in fact {\it isomorphic}\/ to
the ordered differential subfield~$\No(\omega_1)$ of~$\No$  (consisting of  the surreal numbers of countable length).
In~\cite{AvdD} this is also shown 
for maximal smooth   and maximal analytic Hardy fields\endnote{Let us mention here Hamkin's historical thought experiment~\cite{Hamkins}: an early acceptance of CH as a standard set-theoretic axiom alongside the usual ZFC axioms  would presumably have promoted a more wide-spread use of infinitesimals in the style of non-standard analysis, because
CH guarantees the existence of up-to-isomorphism unique saturated elementary extensions of the real field
and its expansions of size $2^{\aleph_0}$. (Without CH there is no such uniqueness.) 
We also note that Cantor~\cite[VI]{Cantor} was hostile to Paul du Bois-Reymond's infinitesimals, which are at the root of Hardy fields and which put  the ``actual infinitesimal'' on as solid a footing as Cantor's ``actual infinite''. Maybe Cantor saw this as unwelcome competition.}
Without even assuming CH, 
one can  embed~$\T$ into each maximal analytic Hardy field and also into~$\No(\omega_1)$, cf.~\cite{AvdD,ADH1+}.
One is left to wonder what G\"odel would have made of this in light of his long standing interest in CH, his appreciation of A.~Robinson's nonstandard analysis, and his  musings,  reported by Conway, whether or not {\it  a solution to the Continuum Hypothesis might yet be possible, but only once the correct theory of infinitesimals had been found}\/~\cite[pp.~209--213]{Roberts}.

\subsection*{Organization of the paper} 
To keep the length of this paper at bay we assume familiarity with the basic setup of [ADH]. (For a brief synopsis which should suffice for reading the present paper see the section {\it Concepts and Results from \textup{[ADH]}}\/ in the introduction to~\cite{ADH4}\endnote{See also \cite{ADH4} or the regularly updated page \url{https://tinyurl.com/ADH-errata} for a list of errata to [ADH].}.) Section~\ref{sec:prelims} has additional definitions and results
from the papers~\cite{ADH5, ADH4,ADH2} used in this note.
Section~\ref{sec:hens} contains a Hensel type Lemma for analytic functions on Hahn fields. This is
applied to solve  certain equations involving power functions in~$\T$, for use in connection with Theorem~B.
We then study differential closure, first  in the general setting of differential fields (Section~\ref{sec:diff closure}), then in  $H$-fields (Sections~\ref{sec:revisiting} and~\ref{sec:diff cl H-fields}),
before proving Theorems~A and~B in Section~\ref{sec:d-cl hardy}.

\subsection*{Notations and conventions} 
For this we follow [ADH]. In particular,  $m$,~$n$   range over the set~$\N=\{0,1,2,\dots\}$ of natural numbers. 
  Given an ordered abelian group~$\Gamma$, additively written, we put $\Gamma^>:=\{\gamma\in\Gamma:\gamma>0\}$.
For an additively written abelian group $A$   set $A^{\ne}:=A\setminus\{0\}$. Given a commutative ring $R$ (always with identity $1$), 
$R^\times$ denotes the multiplicative group of units of~$R$. (So if $K$ is a field, then~$K^{\neq}=K^\times$.)
If $R$ is a differential ring (by convention containing~$\Q$ as a subring) and $y\in R^\times$, then~$y^\dagger=y'/y$ denotes the logarithmic derivative of $y$, so~$(yz)^\dagger = y^\dagger + z^\dagger$ for $y,z\in R^\times$, and
thus  $R^\dagger:=\{y^\dagger:\, y\in R^\times\}$ is an additive subgroup of~$R$. 
 The prefix~``$\d$'' abbreviates ``differentially''; for example, ``$\d$-algebraic'' means ``differentially algebraic''. 

\subsection*{Acknowledgements}
We thank the anonymous referees for helpful remarks which improved the paper.
Joris van der Hoeven has been supported by an ERC-2023-ADG grant for the
ODELIX project (number 101142171). Funded by the European Union. Views and opinions expressed are however those
of the authors only and do not necessarily reflect those of the European
Union or the European Research Council Executive Agency. Neither the European
Union nor the granting authority can be held responsible for them.

\section{Preliminaries}\label{sec:prelims}

\noindent
We begin by recalling definitions, notations, and facts around germs, Hausdorff fields, and Hardy fields  as needed later. Next we briefly discuss
the main result of our paper \cite{ADH2} on the first-order theory of maximal Hardy fields. Finally, we include some   material
on general asymptotic differential algebra from~\cite{ADH4}, before discussing polar coordinates for germs in complexifications of Hardy fields.

\subsection*{Germs}
Let $a$ range over $\R$ and $r$   over $\N\cup\{\omega,\infty\}$. We let $\c^r$ be  the $\R$-algebra of germs at~$+\infty$ of $\R$-valued $\c^r$-functions on half-lines $(a,+\infty)$, for varying $a$, where~$\c^\omega$  means ``analytic''. Thus~$\c:=\c^0$ consists of the germs
at $+\infty$ of continuous   functions~${(a,+\infty)\to\R}$, and
 $$ \c\ =\ \c^0 \ \supseteq \c^1 \ \supseteq\ \c^2  \ \supseteq\  \cdots  \ \supseteq\ \c^\infty  \ \supseteq\ \c^\omega.$$
 The complexification $\c[\imag]=\c+\c\imag$ of $\c$ is the $\C$-algebra consisting of the germs of continuous functions $(a,+\infty)\to\C$, for varying $a$. For $f,g\in \c$ we set $$|f+\imag g|\ :=\ \sqrt{f^2+g^2}\in \c.$$ 
We  have the  $\C$-sub\-al\-ge\-bra~$\Gr[\imag]=\Gr+\Gr\imag$ of $\c[\imag]$.
For~$n\ge 1$ we have the derivation~${g\mapsto g'\colon \Gn[\imag]\to\c^{n-1}[\imag]}$
such that $\text{(germ of $f$)}'=\text{(germ of $f'$)}$ for $\c^n$-functions~$f\colon (a,+\infty)\to\R$, and~${\imag'=0}$. 
Therefore~$\Gi[\imag]  := \bigcap_{n}\, \Gn[\imag]$ 
is naturally a differential ring with ring of constants~$\C$, and
$\Gi  := \bigcap_{n}\, \Gn$
is a differential subring of $\Gi[\imag]$ with ring of constants $\R$.
Note that~$\Gi[\imag]$ has~$\Ginf[\imag]$ as a differential subring, $\Gi$ has $\Ginf$ as a differential subring,
and  $\Ginf$ has in turn
the differential subring~$\Gom$.

\subsection*{Asymptotic relations}
We often use the same notation for a $\C$-valued function 
on a subset of~$\R$ containing an interval $(a, +\infty)$   as for its germ if the resulting ambiguity is harmless.
We equip $\c$ with the partial ordering given by~$f\leq g:\Leftrightarrow f(t)\leq g(t)$ for all sufficiently large real $t$,
and equip $\c[\imag]$ with the asymptotic relations~$\preceq$,~$\prec$,~$\sim$ defined as follows: for $f,g\in \c[\imag]$,
\begin{align*} f\preceq g\quad &:\Longleftrightarrow\quad \text{$|f|\le c|g|$ for some  $c\in \R^{>}$,}\\
f\prec g\quad &:\Longleftrightarrow\quad   \text{$g\in \c[\imag]^\times$ and $\abs{f}\leq c\abs{g}$ for all $c\in\R^>$},\\
f\sim g\quad &:\Longleftrightarrow\quad   f-g\prec g.
\end{align*}

\subsection*{Hausdorff fields}
Let   $H$ be a {\it Hausdorff field}\/:
a subfield of $\mathcal C$. 
Then the partial ordering of $\c$ restricts to a total ordering on $H$ which makes~$H$ into an ordered field.
The ordered field $H$ has a convex subring~${\mathcal{O}:=\{f\in H:f\preceq 1\}}$,
which is a valuation ring of $H$, and we
consider~$H$ accordingly as a valued ordered field. 
Moreover, $H[\imag]$ is a subfield of $\c[\imag]$, and~$\mathcal{O}+\mathcal{O}\imag = \big\{f\in H[\imag]: {f\preceq 1}\big\}$ is the unique valuation ring of $H[\imag]$ whose intersection with $H$ is~$\mathcal{O}$. In this way we consider~$H[\imag]$ as a valued field extension of $H$. 
The asymptotic relations~$\preceq$,~$\prec$,~$\sim$ on~$\c[\imag]$
restrict to the asymptotic relations $\preceq$, $\prec$, $\sim$ on $H[\imag]$ that~$H[\imag]$ has as a valued field (cf.~[ADH, (3.1.1)]; likewise with $H$ in place of $H[\imag]$.

\subsection*{Hardy fields}
Let $H$ be a {\it Hardy field}\/:  a   differential subfield of~$\Gi$.   Then $H$ is a Hausdorff field, and
we consider~$H$  as an ordered valued differential field with ordering and valuation as above.
Any Hardy field is a pre-$H$-field; if it contains $\R$, it is is an $H$-field.
We   equip the differential subfield $H[\imag]$  of~$\Calinf[\imag]$ with the unique valuation ring 
lying over that of~$H$. Then $H[\imag]$ is a pre-$\d$-valued field of $H$-type with small derivation, and
if~$H\supseteq\R$, then $H[\imag]$ is $\d$-valued   with constant field $\C$.

Recall that $H$ is said to be {\it maximal}\/ if it has no proper Hardy field extension, and that every Hardy field has a 
maximal Hardy field extension. The intersection $\Ex(H)$ of all maximal Hardy field extensions of $H$ is a Hardy field extension of $H$,
called the {\it perfect hull} of $H$, and if $\Ex(H)=H$, then $H$ is said to be {\it perfect.}\/ We also say that~$H$ is
{\it $\d$-maximal}\/ if 
it has no proper $\d$-algebraic Hardy field extension.   Zorn yields a
$\d$-maximal $\d$-algebraic Hardy field extension of $H$, hence the intersection $\Dx(H)$ of all $\d$-maximal Hardy fields containing $H$  is a $\d$-algebraic Hardy field extension of~$H$, called the {\it $\d$-perfect hull}\/ of $H$.  We  call   $H$    {\it $\d$-perfect}\/ if~$\Dx(H)=H$. 
If $H$ is $\d$-perfect, then $H\supseteq\R$ and $H$ is a Liouville closed $H$-field, by~\cite[remarks after Proposition~4.2]{ADH5}.
We have $\Dx(H)\subseteq\Ex(H)$; indeed, by \cite[Lemma~4.1]{ADH5}:

\begin{lemma}\label{lem:Dx Ex}
$\Dx(H)=\big\{f\in\Ex(H):\text{$f$ is $\d$-algebraic over $H$}\big\}$.
\end{lemma}

\noindent
A {\em smooth Hardy field}\/ is a Hardy field $H\subseteq \Ginf$, and an {\em analytic Hardy field\/}   is a Hardy field~${H\subseteq \Gom}$.
Instead of smooth and analytic Hardy fields we also speak of $\Ginf$- and $\Gom$-Hardy fields. Let $r\in\{\infty,\omega\}$.
A {\it $\c^r$-maximal  Hardy field}\/ is a $\c^r$-Hardy field which has no proper $\c^r$-Hardy field extensions.
If $H\subseteq\c^r$, then we let~$\Ex^r(H)$ be the intersection of all $\c^r$-maximal Hardy fields containing $H$.
Thus using the notation from the introduction, our main objects of interest in this paper are $\Ex=\Ex(\Q)$, $\Ex^\infty=\Ex^\infty(\Q)$, and~$\Ex^\omega=\Ex^\omega(\Q)$.
Instead of ``$\c^\infty$-maximal'' and ``$\c^r$-maximal'' we also write ``maximal smooth'' and ``maximal analytic'', respectively.
The following is~\cite[Co\-rol\-la\-ry~7.8]{ADH2}:

\begin{prop}\label{prop:Hardy field ext smooth}
Suppose $H$ is smooth. Then every $\d$-algebraic Hardy field extension of $H$ is also smooth; in particular, $\Dx(H)$ is smooth.  Likewise with ``smooth'' replaced by ``analytic''.
\end{prop}

\noindent
Let now $H$ be a $\Ginf$-Hardy field.  Then by Proposition~\ref{prop:Hardy field ext smooth},
$H$ is $\d$-maximal iff~$H$ has no proper $\d$-algebraic $\Ginf$-Hardy field extension; thus  every $\Ginf$-maximal Hardy field is $\d$-maximal, and $H$ has a $\d$-maximal $\d$-algebraic $\Ginf$-Hardy field extension.   The same remarks apply with $\omega$ in place of $\infty$. 

\subsection*{The main result of \cite{ADH2}} 
A {\it closed $H$-field}\/ (or {\it $H$-closed field}\/) is a Liouville closed, $\upo$-free,   newtonian $H$-field. 
By [ADH], the closed $H$-fields with small derivation are precisely the models of the elementary theory of $\T$ as an ordered valued differential field.
Hence by the next theorem, every $\d$-maximal Hardy field as an ordered valued differential field is elementarily equivalent to $\T$. 

\begin{theorem}\label{thm:maxhardymainthm}
For a Hardy field $H$, the following are equivalent:
\begin{enumerate}
\item[\textup{(i)}] $H$ is a $\d$-maximal Hardy field;
\item[\textup{(ii)}] $H\supseteq \R$ and $H$ is a closed $H$-field;
\item[\textup{(iii)}] $H\supseteq \R$ and $H$ is a Liouville closed $H$-field having \textup{DIVP}.
\end{enumerate}
\end{theorem}

\noindent
Here \cite[Theorem~11.19]{ADH2} is the equivalence (i)~$\Leftrightarrow$~(ii), and (ii)~$\Leftrightarrow$~(iii) is \cite[Corollary~1.7]{ADHip}.
The preceding theorem and Proposition~\ref{prop:Hardy field ext smooth} yield \cite[Corollary~11.20]{ADH2}:

\begin{cor}\label{thm:extend to H-closed}
Each Hardy field $H$ has a $\d$-algebraic $H$-closed Hardy field extension.
If $H$ is a $\Ginf$-Hardy field, then so is any such extension, and likewise with~$\Gom$ in place of $\Ginf$.
\end{cor}

\subsection*{Logarithmic derivatives}
Let $K$ be a differential field. 
The {\it group of logarithmic derivatives}\/ of $K$ is the additive subgroup~$K^\dagger=\{f^\dagger:f\in K^\times\}$  of~$K$.
If $K$ is algebraically closed or real closed, then $K^\dagger$ is divisible.
Here is  \cite[Lemma~1.2.1]{ADH4}:

\begin{lemma}\label{lem:Kdagger alg closure}
Suppose $K^\dagger$ is divisible,  $L$ is a differential field extension of~$K$ such that~${L^\dagger\cap K=K^\dagger}$, and
$M$ is a differential field extension of $L$ and algebraic over~$L$. Then~$M^\dagger\cap K=K^\dagger$.
\end{lemma}

\noindent
Suppose that $H$ is a real closed asymptotic field whose valuation ring $\mathcal{O}$ is convex with respect to the ordering of $H$, 
and $K:=H[\imag]$. Then $\mathcal{O}_K=\mathcal{O}+\mathcal{O}\imag$ is the unique valuation ring of $K$ with $\mathcal O_K\cap H=\mathcal O$
[ADH, 3.5.15].
Equipped with this valuation ring, $K$ is an asymptotic field
extension of $H$ [ADH,  9.5.3], and if $H$ is $H$-asymptotic, then so is $K$.
With $\wr(a,b):=ab'-a'b$ (the wronskian of $a,b$), set
$$S\ :=\ \big\{y\in K:\, |y|=1\big\}, \qquad
W\ :=\ \big\{\!\wr(a,b):\, a,b\in H,\ a^2+b^2=1\big\}.$$ 
Then $S$ is a subgroup of $\mathcal{O}_K^\times$ with $S^\dagger = W\imag$ and 
$K^\dagger=H^\dagger\oplus W\imag$ by \cite[Lemma~1.2.4]{ADH4}. Recall from [ADH, 14.2] that for 
asymptotic fields $E$ (such as $H$, $K$) we defined $$\I(E)\ :=\ \{f\in E:\ f\preceq g' \text{ for some }g\preceq 1 \text{ in }E\}.$$  
Since~$\der\mathcal O\subseteq \I(H)$,  
we also have $W\subseteq \I(H)$, and thus: $W = \I(H) \ \Longleftrightarrow\  \I(H)\imag\subseteq K^\dagger$. 
Moreover, by \cite[Lemma~1.2.13]{ADH4} we have 
$W=\I(H)\subseteq H^\dagger \ \Longleftrightarrow\ 
\I(K)\subseteq K^\dagger$.

\begin{lemma}\label{lem:logder real closed}
Suppose $H$ is $H$-asymptotic with asymptotic integration, and $K$ is $1$-linearly newtonian. Then $K^\dagger=H^\dagger\oplus\I(H)\imag$.
Moreover, if $F$ is   a real closed asymptotic extension of $H$ whose valuation ring is convex,  
then $$F[\imag]^\dagger\cap K\ =\ (F^\dagger\cap H)\oplus\I(H)\imag.$$ 
\end{lemma}

\begin{proof}
By \cite[Corollary~1.2.14]{ADH4} we have $\I(K)\subseteq K^\dagger$ and thus $W=\I(H)$ and  $K^\dagger=H^\dagger\oplus\I(H)\imag$ by the remarks before the lemma. The second part of the lemma follows from this and \cite[Corollary~1.2.15]{ADH4}.
\end{proof}
 
\subsection*{The universal exponential extension}
{\it In this subsection $K$ is a differential field with algebraically closed constant field $C$ and divisible group $K^\dagger$
 of logarithmic derivatives.}\/
An {\it exponential extension}\/ of $K$ is   a differential ring extension~$R$ of~$K$ such that~$R=K[E]$ for some $E\subseteq R^\times$ with~${E^\dagger\subseteq K}$.
By~\cite[Section~2.2, especially Corollary~2.2.11 and remarks preceding it]{ADH4},
there is an exponential extension $\Univ$ of $K$ with~$C_{\Univ}=C$ such that every exponential extension $R$ of $K$ with $C_R=C$
embeds into~$\Univ$ over $K$; any two such exponential extensions of $K$ are isomorphic over $K$.
We call $\Univ$ the {\it universal exponential extension}\/ of $K$,   denoted by $\Univ_K$ if we want to stress the dependence on~$K$.
By its construction in \cite[Section~2.2]{ADH4}, $\Univ$ is an integral domain,
and $(\Univ^\times)^\dagger =  K$ by \cite[remarks before Example~2.2.4]{ADH4}.
 We denote the differential fraction field of $\Univ$ by $\Omega$ (or $\Omega_K$); then   $C_\Omega=C$ by \cite[remark before Lemma~2.2.7]{ADH4}. 
If $L$ is a differential field extension of $K$ such that $C_L$ is algebraically closed and $L^\dagger$ is divisible, then
by \cite[Lemma~2.2.12]{ADH4} the natural inclusion~$K\to L$ extends to an embedding $\Univ_K\to\Univ_L$ of differential rings.

Suppose that $K$ is $\d$-valued of $H$-type with $\Gamma\neq\{0\}$ and with small derivation. 
By \cite[Lemma~2.5.1]{ADH4}, the  valuation of $K$ extends to a valuation on~$\Omega$ that makes~$\Omega$ a $\d$-valued extension of $K$ of $H$-type with small derivation, called a {\it spectral extension}\/
of the valuation of $K$ to $\Omega$.
By \cite[Lemma~2.5.3 and Corollary~2.5.5]{ADH4} we have:

\begin{lemma} \label{lem:as cpl Omega}
If $K$ is $\upl$-free and $\I(K)\subseteq K^\dagger$, then the $H$-asymptotic couple~$(\Gamma_\Omega,\psi_\Omega)$ of $\Omega$ equipped with a spectral
extension of the valuation of $K$ is closed  with $\Psi_\Omega:=\big\{ \psi_\Omega(\gamma):\gamma\in\Gamma_\Omega^{\neq}\big\}\subseteq\Gamma$.
\end{lemma}

\subsection*{Some facts about complexified Hardy fields}

{\it In this subsection, $H$ is a Hardy field.}\/
The $H$-asymptotic field extension $K:=H[\imag]$ of~$H$ is a differential subring of~$\Calinf[\imag]$.  
The next proposition, \cite[Proposition~6.11]{ADH5}, considers the condition~${\I(K)\subseteq K^\dagger}$ in this setting:

\begin{prop}\label{prop:cos sin infinitesimal, 2}
Suppose $H\supseteq\R$ is closed under integration.  
Then the following   are equivalent:
\begin{enumerate}
\item[\textup{(i)}] $\I(K)\subseteq K^\dagger$;
\item[\textup{(ii)}] $\ex^{f}\in K$ for all $f\in K$ with $f\prec  1$;
\item[\textup{(iii)}] $\ex^\phi,\cos \phi, \sin\phi\in H$ for all $\phi\in H$ with $\phi\prec 1$.
\end{enumerate}
\end{prop}

\noindent
Next we discuss ``polar coordinates'' of nonzero elements of $K$:

\begin{lemma}\label{lem:arg}
Let $f\in\c[\imag]^\times$. Then $\abs{f}\in\c^\times$, and $f=\abs{f}\ex^{\phi\imag}$ for some $\phi\in\c$.
Such~$\phi$ is unique up to addition of
an element of $2\pi\Z$. If also $f\in\c^r[\imag]^\times$, $r\in\N\cup\{\infty,\omega\}$, then $\abs{f}\in \c^r$
and $\phi\in\c^r$ for such $\phi$.
\end{lemma}
\begin{proof}
It is clear that  $\abs{f}\in \c^r$ if $f\in\c^r[\imag]^\times$. To show existence of~$\phi$ we replace~$f$ by~$f/\abs{f}$
to arrange $\abs{f}=1$.
Take $a\in\R$ and a continuous representative~${[a,+\infty)\to \C}$ of $f$, also denoted by $f$, such that $\abs{f(t)}=1$ for all $t\geq a$.
The proof of \cite[(9.8.1)]{Dieudonne} shows that for $b\in (a,+\infty)$ and $\phi_a\in\R$ with
$f(a)=\ex^{\phi_a\imag}$
there is a unique continuous function
$\phi\colon [a,b]\to\R$ such that $\phi(a)=\phi_a$ and $f(t) = \ex^{\phi(t)\imag}$ for all $t\in [a,b]$, and if also $f|_{[a,b]}$ is of class
$\c^1$, then so is this $\phi$  with $\imag\phi'(t) = f'(t)/f(t)$ for all~$t\in[a,b]$. 
With $b\to +\infty$ this yields the desired result. 
\end{proof}

\begin{lemma}\label{lem:fexphii}
Suppose $H\supseteq\R$ is Liouville closed  and $f\in \c^1[\imag]^\times$. Then~$f^\dagger\in K$ iff~$\abs{f}\in H^>$ and  $f=\abs{f}\ex^{\phi\imag}$ for some $\phi\in H$.
If in addition $f\in K^\times$, then $f=\abs{f}\ex^{\phi\imag}$ for some $\phi\preceq 1$ in $H$.
\end{lemma}
\begin{proof}
Take $\phi\in\c$ as in Lemma~\ref{lem:arg}. Then $\phi\in\c^1$ and $\Re f^\dagger=\abs{f}^\dagger$, $\Im f^\dagger=\phi'$.
If~$f\in K^\times$, then
the remarks preceding Lemma~\ref{lem:logder real closed} give $\phi'\in \I(H)$,  so~${\phi\preceq 1}$. 
\end{proof}
 
\begin{cor}\label{cor:fexphii}
Suppose $H\supseteq\R$ is Liouville closed with $\I(K)\subseteq K^\dagger$. Let $L$ be a differential subfield of $\Calinf[\imag]$ containing $K$.
Then $L^\dagger \cap K = K^\dagger$.
\end{cor}
\begin{proof}
Let $f\in L^\times$ satisfy $f^\dagger \in K$. Then $f=\abs{f}\ex^{\phi\imag}$ with $\abs{f}\in H^>$, $\phi\in H$,
by Lem\-ma~\ref{lem:fexphii}. Hence $\ex^{\phi\imag}, \ex^{-\phi\imag}\in L$, thus
$\cos \phi = \frac{1}{2}(\ex^{\phi\imag}+\ex^{-\phi\imag})\in L$. In particular, $\cos \phi$ doesn't oscillate, so $\phi\preceq 1$. Thus $f=|f|(\cos \phi + \imag \sin \phi)\in K$ by Proposition~\ref{prop:cos sin infinitesimal, 2}.
\end{proof}

\section{Hensel's Lemma for Analytic Functions on Hahn Fields}\label{sec:hens}

\noindent
Let $\k$ be a field and $\fM$ a (multiplicatively written) ordered abelian group, with the ordering of $\fM$ denoted by $\preceq$.  
Let $K=\k[[\fM]]$ be the corresponding Hahn field over $\k$. Its (Hahn) valuation has valuation ring $\mathcal{O}:=\k[[\fM^{\preceq 1}]]$, with
maximal ideal~$\smallo:=\k[[\fM^{\prec 1}]]$. 
Let~$Q(Z)$ be a power series $\sum_n a_nZ^n$ with all $a_n\in \mathcal{O}$. Then for $z\in \smallo$  the sum~$\sum_n a_nz^n$ exists with value $Q(z):=\sum_n a_nz^n\in \mathcal{O}$. Let~$Q'(Z):=\sum_n ({n+1})a_{n+1}Z^n$ be the formal derivative of $Q(Z)$.  Here is a Hensel type lemma: 

\begin{lemma}\label{lem:hens}
If $Q(0)\prec 1$ and $Q'(0)\asymp 1$, then $Q(z)=0$ for a unique $z\in \smallo$.
\end{lemma}

\begin{proof} Assume $Q(0)\prec 1$, $Q'(0)\asymp 1$. This means $a_0\prec  1$ and $a_1\asymp 1$. Multiplying with $a_1^{-1}$ we reduce to the case $a_1=1$,  so for $z\in \smallo$ we have
$$Q(z)=0\ \Longleftrightarrow\ z-Q(z)=-a_0-a_2z^2-a_3z^3- \cdots= z.$$ Now for distinct $z_1, z_2\in \smallo$ we have
$z_1^n-z_2^n\prec z_1-z_2$ for $n\ge 2$, so the map~$z\mapsto {z-Q(z)}\colon \smallo \to\smallo$ is contracting:
$\big(z_1-Q(z_1)\big)-\big(z_2-Q(z_2)\big) \prec z_1-z_2$ for distinct~${z_1, z_2\in \smallo}$. Thus this  map has a unique fixpoint  $z\in \smallo$ by [ADH, 2.2.12]. 
\end{proof}

\noindent
The valued field $\T$ of transseries is not a Hahn field, but it is a
direct union of Hahn subfields over the coefficient field $\R$, and in
this way the above lemma applies to $\T$.
For $c\in\R$ and $f\in\T^>$ we set $f^c:=\exp(c\log f)\in \T$. Then~$(f^c)^\dagger=cf^\dagger$ and
$z\in\T$ with $z\prec 1$ gives~$(1+z)^c=\sum_n {c\choose n}z^n\in\T$: for this, note that
for real $t\in (-1,1)$,
$$(1+t)^c\ =\ \sum_{n=0}^\infty {c\choose n}t^n\ =\ \ex^{c\log(1+t)}\ =\ 
\sum_{n=0}^\infty \frac{c^n}{n!}\left(\sum_{m=1}^\infty \frac{(-1)^{m-1}t^m}{m}\right)^n,$$
so the formal power series $\sum_{n=0}^\infty {c\choose n}Z^n$ and $\sum_{n=0}^\infty \frac{c^n}{n!}\big(\sum_{m=1}^\infty \frac{(-1)^{m-1}Z^m}{m}\big)^n$ in $Z$ are equal, from which the identity claimed about $z$ follows by substitution. 


\begin{cor}\label{cor:hens, 2}
Let $c\in\R\setminus\{-1\}$ and $\varepsilon\in\smallo_{\T}$. Then there is a unique $z\in\smallo_{\T}$ with
$$(1+z)^c \cdot ( 1+\varepsilon +z)\ =\ 1.$$
\end{cor}
\begin{proof} Note that  $Q(Z):= \big(\sum_n{c\choose n}Z^n\big)(1+\varepsilon+Z)-1$ satisfies the
assumption of Lem\-ma~\ref{lem:hens}: its constant term is $\varepsilon$ and its term of degree $1$  is~$(1+c+c\varepsilon)Z$.
\end{proof}

\section{Relative Differential Closure}\label{sec:diff closure}

\noindent
Let $K\subseteq L$ be an extension of differential fields, and let $r$ range over $\N$. We say that~$K$ is {\bf $r$-differentially closed} in $L$ 
for every~$P\in K\{Y\}^{\neq}$ of order~$\leq r$, each zero of $P$ in~$L$   lies in~$K$. 
We also say that $K$ is {\bf weakly $r$-differentially closed} in~$L$ if 
 every~$P\in K\{Y\}^{\ne}$ of order~$\leq r$ with a zero in~$L$ has a zero in~$K$.
 We abbreviate ``$r$-differentially closed'' by ``$r$-$\d$-closed.''
 Thus 
$$\text{$K$ is $r$-$\d$-closed in $L$}\quad\Longrightarrow\quad\text{$K$ is  weakly $r$-$\d$-closed in~$L$,}$$ 
and
\begin{align}\label{eq:0-d-closed}
\text{$K$ is $0$-$\d$-closed in $L$}&\quad\Longleftrightarrow\quad 
  \text{$K$ is weakly $0$-$\d$-closed in $L$} \\ &\quad\Longleftrightarrow\quad
  \text{$K$ is algebraically closed in $L$.} \notag
\end{align} 
Hence, with $C$ the constant field of $K$ and $C_L$ the constant field of $L$:
\begin{equation}\label{eq:weakly 0-d-closed}
\text{$K$ is weakly $0$-$\d$-closed in $L$}\quad\Longrightarrow\quad\text{$C$ is algebraically closed in $C_L$.}
\end{equation}
Also, if~$K$ is weakly $0$-$\d$-closed in $L$
and $L$ is algebraically closed, then $K$ is algebraically closed, and similarly
with ``real closed'' in place of ``algebraically closed''. In [ADH, 5.8] we defined
$K$ to be {\it  weakly $r$-$\d$-closed}\/ if every $P\in K\{Y\}\setminus K$ of~order~$\leq r$ has a zero in $K$.
Thus if $K$ is weakly $0$-$\d$-closed, then $K$ is algebraically closed, and
$$\text{$K$ is weakly $r$-$\d$-closed}
\ \Longleftrightarrow\  \begin{cases} &\parbox{21em}{$K$ is weakly $r$-$\d$-closed in every differential field extension of $K$.}\end{cases}$$
If
$K$ is weakly   $r$-$\d$-closed in $L$, then $P(K)=P(L)\cap K$ for all $P\in K\{Y\}$ of or\-der~$\leq r$;
in particular, 
\begin{equation}\label{eq:weakly 1-d-closed}
\text{$K$ is weakly $1$-$\d$-closed in $L$}\quad\Longrightarrow\quad \der K=\der L\cap K.
\end{equation}
Also, 
\begin{equation}\label{eq:1-d-closed}
\text{$K$ is $1$-$\d$-closed in $L$}\quad\Longrightarrow\quad C=C_L\text{ and } K^\dagger=L^\dagger\cap K.
\end{equation}
Moreover:

\begin{lemma}\label{lem:weakly r-d-closed}
Suppose $K$ is weakly $r$-$\d$-closed in $L$. If $L$ is $r$-linearly surjective, then so is $K$, and if $L$
is
$(r+1)$-linearly closed, then so is $K$.
\end{lemma}
\begin{proof}
The first claim from the remarks preceding the lemma, and the proof of the second statement is like that of  [ADH, 5.8.9].
\end{proof}

\noindent
Sometimes we get more than we bargained for:

\begin{lemma}
Suppose $K$ is not algebraically closed, $C\ne K$, and
$K$ is weakly $r$-$\d$-closed in $L$. Let $Q_1,\dots,Q_m\in K\{Y\}^{\neq}$ of order~$\leq r$ have a common zero
in~$L$,~$m\ge 1$. Then they have a common zero in $K$.
\end{lemma}
\begin{proof}
We claim that some polynomial $\Phi\in K[X_1,\dots,X_m]$ has $(0,\dots,0)\in K^m$ as its only
zero in $K^m$.
(This is folklore, but we didn't find a proof in the literature.)
The claim is clear for $m=1$. Since $K$ is not algebraically closed, we have a univariate polynomial $f\in K[X]$ of degree $>1$ 
 without a zero in $K$. 
Then the homogenization  $g(X,Y)\in K[X,Y]$
 of~$f(X)$ has $(0,0)$ as its only zero in $K^2$. 
This proves the claim for $m=2$. Now use induction on $m$: if $\Phi$ has the above property, then
 $g\big(\Phi(X_1, \ldots, X_{m}), X_{m+1}\big)$ has~$(0,\dots, 0,0)\in K^{m+1}$ as its only zero in $K^{m+1}$.

By the claim, the differential
polynomial~$P:=\Phi(Q_1,\dots,Q_m)\in K\{Y\}$ is nonzero (use [ADH, 4.2.1]) and
has order~$\leq r$. For $y\in L$ we have $$Q_1(y)=\cdots=Q_m(y)=0\quad \Longrightarrow\quad P(y)=0,$$
and for $y\in K$ the converse of this implication also holds. 
\end{proof}

\noindent
Here is a characterization of $r$-$\d$-closedness:

\begin{lemma}
The following are equivalent:
\begin{enumerate}
\item[\textup{(i)}] $K$ is $r$-$\d$-closed in $L$;
\item[\textup{(ii)}] there is no differential subfield $E$ of $L$ with $K\subset E$ and $\operatorname{trdeg}(E|K)\le r$. 
\end{enumerate}
\end{lemma}
\begin{proof}
Let $E$ be a differential subfield of $L$   with $E\supseteq K$ and $\operatorname{trdeg}(E|K)\le r$. Then for $y\in E$ we have $\operatorname{trdeg}(K\langle y\rangle|K)\le \operatorname{trdeg}(E|K)\le r$,
so~[ADH, 4.1.11]  gives a~$P\in K\{Y\}^{\ne}$ of order~$\le r$
with $P(y)=0$, hence $y\in K$ if $K$ is $r$-$\d$-closed in $L$. This shows (i)~$\Rightarrow$~(ii). For the converse, note that for $P\in K\{Y\}^{\neq}$ of order~$\le r$ and~$y\in L$ with~$P(y)=0$ we have
$K\langle y\rangle=K(y,y',\dots,y^{(r)})$, so $\operatorname{trdeg}(K\langle y\rangle|K)\le r$.
\end{proof}

\begin{cor}\label{transdclosed}
If $K$ is $r$-$\d$-closed in $L$, then $K$ is $r$-$\d$-closed in each differential subfield of $L$ containing $K$.
If $M$ is a differential field extension of $L$ such that $L$ is $r$-$\d$-closed in $M$ and $K$ is $r$-$\d$-closed in $L$, then $K$ is $r$-$\d$-closed in $M$.
\end{cor}

\noindent
We say that $K$ is {\bf differentially closed}  in~$L$ if $K$ is 
$r$-$\d$-closed in~$L$ for each~$r$, and similarly we define when  $K$ is {\bf weakly differentially closed}  in~$L$.
We also use ``$\d$-closed'' to abbreviate ``differentially closed''.
If~$K$, as a differential ring, is existentially closed in~$L$, then $K$ is weakly
$\d$-closed in~$L$.
The elements of $L$ that are $\d$-algebraic over $K$ form
the smallest differential subfield of $L$ containing~$K$ which is  $\d$-closed in $L$; we call it the
{\bf differential closure} (``$\d$-closure'' for short) of~$K$ in $L$.  Thus $K$ is $\d$-closed in $L$ iff   no   $\d$-subfield of $L$
properly containing~$K$ is
 $\d$-algebraic over $K$. For example, if
    $L=K\langle B\rangle$ where~$B\subseteq K$ is $\d$-algebraically independent over $K$ [ADH, p.~205], then $K$ is $\d$-closed in $L$ by Corollary~\ref{transdclosed} and~[ADH, 4.1.5].
This notion of being differentially closed does not seem prominent in the differential algebra literature, though the definition occurs (as ``differentially algebraic closure'') in~\cite[p.~102]{Kolchin-DA}. Here is a useful fact about it: 

\begin{lemma}\label{lem:dc base field ext}
Let $F$ be a differential field extension of $L$ and let $E$ be a subfield of~$F$ containing $K$ such that $E$ is algebraic over $K$ and $F=L(E)$. 
\[
\xymatrix@R-1.5pc@C-1.5pc{& F=L(E) & \\
E \ar@{-}[ur] & & L \ar@{-}[ul] \\
& K \ar@{-}[ul] \ar@{-}[ur] & } 
\]
Then $K$ is $\d$-closed in $L$ iff $E\cap L = K$ and $E$ is $\d$-closed in $F$.
\end{lemma}

\begin{proof}
Suppose $K$ is $\d$-closed in $L$. Then $K$ is algebraically closed in $L$, so $L$ is linearly disjoint from~$E$ over $K$. (See \cite[Chapter~VIII, \S{}4]{Lang}.) In particular~${E\cap L=K}$.
Now let $y\in F$ be $\d$-algebraic over $E$; we  claim that $y\in E$. Note that $y$ is $\d$-algebraic over~$K$. 
Take a field extension $E_0\subseteq E$ of $K$ with $[E_0:K]<\infty$ (so $E_0$ is a $\d$-subfield of~$E$) such that $y\in L(E_0)$;
replacing~$E$,~$F$ by $E_0$, $L(E_0)$, respectively, we arrange that $n:=[E:K]<\infty$.
Let $b_1,\dots,b_n$ be a basis of the $K$-linear space~$E$;
then~$b_1,\dots,b_n$ is also a basis of the $L$-linear space $F$.
Let $\sigma_1,\dots,\sigma_n$ be the distinct field embeddings~$F\to L^{\operatorname{a}}$ over $L$.
Then the vectors
$$\big(\sigma_1(b_1),\dots,\sigma_1(b_n)\big),\dots, \big(\sigma_n(b_1),\dots,\sigma_n(b_n)\big)\in (L^{\operatorname{a}})^n$$
are $L^{\operatorname{a}}$-linearly independent \cite[Chapter~VI, Theorem~4.1]{Lang}. 
Let $a_1,\dots,a_n\in L$ be such that $y=a_1b_1+\cdots+a_nb_n$. Then
$$\sigma_j(y) = a_1\sigma_j(b_1)+\cdots+a_n\sigma_j(b_n)\quad\text{for $j=1,\dots,n$,}$$
hence by Cramer's Rule, 
$$a_1,\dots,a_n\in K\big(\sigma_j(y),\sigma_j(b_i):i,j=1,\dots,n\big).$$
Therefore $a_1,\dots,a_n$ are $\d$-algebraic over $K$, since    $\sigma_j(y)$ and $\sigma_j(b_i)$ for $i,j=1,\dots,n$ are.
Hence $a_1,\dots,a_n\in K$ since $K$ is $\d$-closed in $L$, so $y\in E$ as claimed. This shows the forward implication. The
backward direction is clear.
\end{proof}

\begin{cor}\label{cor:dc base field ext}
If $-1$ is not a square in $L$ and $\imag$ in a differential field extension of~$L$ satisfies $\imag^2=-1$,
then:  $K$ is $\d$-closed in $L\ \Leftrightarrow\ K[\imag]$ is $\d$-closed in $L[\imag]$.
\end{cor} 

\noindent
The notion of $\d$-closedness concerns one-variable differential polynomials, but under extra assumptions on the
differential field extension $K\subseteq L$, it has consequences for systems of differential polynomial (in-)equalities in several indeterminates, by the next lemma. First some notation: for $P\in K\{Y\}$, $Y$ a single indeterminate, we let~$\Zero(P)$ denote the set of zeros of $P$ in $K$:
$$\Zero(P)\ =\ \Zero_K(P)\  :=\  \big\{y\in K: P(y)=0\big\}.$$
Let $S\subseteq K^n$ be defined in $K$ by a quantifier-free formula~$\phi(y_1,\dots,y_n)$
in the language of differential fields with names for the elements of $K$, and let 
$S_L\subseteq L^n$ be defined
in~$L$ by the same formula. Then $S=S_L\cap K^n$, and
if~$K$ is existentially closed in~$L$, then $S_L$ does not depend on the choice of $\phi$.  

Suppose $C\ne K$. Then we defined in \cite[Section~1]{ADHdim}
for any set $S\subseteq K^n$ its (differential) dimension $\dim S\in \{-\infty, 0,1,2,\dots,n\}$, with $\dim S=-\infty$ iff $S=\emptyset$. 
If $S\ne\emptyset$ and~$S$ is finite, then $\dim S=0$, and if $S\subseteq K$, then $\dim S=0$
iff~$S\ne\emptyset$ and~$S\subseteq \Zero(P)$ for some~${P\in K\{Y\}^{\ne}}$.

\begin{lemma}\label{lem:dim0}
Suppose that $C\ne K$, that $K$ is both existentially closed in $L$
and $\d$-closed in $L$, and that~$\dim S=0$. Then $S=S_L$.
\end{lemma}
\begin{proof}
Since $K$ is existentially closed in $L$,   [ADH, B.8.5] yields a differential field  extension $K^*$ of $L$ such that $K\preceq K^*$. 
For $i=1,\dots,n$ and the $i$th coordinate projection $\pi_i\colon K^n\to K$ we have~$\dim \pi_i(S)=0$
and therefore $\pi_i(S)\subseteq\Zero(P_i)$ with~$P_i\in K\{Y\}^{\ne}$, by \cite[Lemma~1.2]{ADHdim}, so~$S\subseteq\Zero(P_1)\times\cdots\times\Zero(P_n)$, hence 
\begin{align*}
S_{K^*}&\subseteq \Zero_{K^*}(P_1)\times\cdots\times\Zero_{K^*}(P_n)&&\text{since $K\preceq K^*$, and thus} \\
S_L=S_{K^*}\cap L^n&\subseteq \Zero_{L}(P_1)\times\cdots\times\Zero_{L}(P_n)\subseteq K^n&&\text{since~$K$ is $\d$-closed in $L$.}
\end{align*}
Thus $S=S_L$.
\end{proof}

\noindent
Recall from [ADH, 4.7]  that $K$ is said to be {\it $\d$-closed}\/ if for all $P\in K[Y,\dots, Y^{(r)}]^{\ne}$ and
$Q\in K[Y,\dots, Y^{(r-1)}]^{\ne}$ such that $Y^{(r)}$ occurs in $P$
there is a~$y\in K$ such that~$P(y)=0$ and $Q(y)\ne 0$. (For $r=0$ this says that $K$ is algebraically closed.) 
If~$K$ is $\d$-closed in~$L$ and $L$ is $\d$-closed, then $K$ is $\d$-closed. Therefore, since
  $K$ has a $\d$-closed differential field extension
(cf.~[ADH, remark after 4.7.1]): if $K$ is $\d$-closed in each  differential field extension, then $K$ is $\d$-closed.
The reverse implication, however, is false:
if $C$ is algebraically closed,
then $K$ has a proper $\d$-algebraic differential field extension $L$, even with~${C_L=C}$. To see this we use the 
theorem of Rosenlicht~\cite[Proposition~2]{Rosenlicht74} below.  We equip the field $C(Y)$ of rational functions over the constant field $C$ of $K$ with the derivation~$\partial=\partial/\partial Y$~[ADH, p.~200].

\begin{theorem}[Rosenlicht]\label{thm:RosenlichtOrder1}
Let $R\in C(Y)^\times$ be such that $1/R$ is neither of the form~$a\,\frac{\partial U}{U}$ $(a\in C, U\in C(Y)^\times)$, nor of the form
$\partial V$ $(V\in C(Y)^\times)$.
Let $y$ be an element of a differential field extension of $K$ such that $y$ is transcendental over $K$ and $y'=R(y)$.   Then $C_{K\<y\>}=C$.
\end{theorem}

\noindent
 Let $c\in C$ and consider the differential polynomial $P=P_c\in K\{Y\}$ and the rational function $R=R_c\in C(Y)$ given by 
$$P(Y) := Y'\big( c(Y+1)+Y \big) - Y(Y+1)\in K\{Y\},\qquad R(Y):=\frac{Y(Y+1)}{c(Y+1)+Y}\in C(Y).$$
For $y\notin K$ in a differential field extension of $K$  we have $P(y)=0$ iff~${y'=R(y)}$.
We identify $\Q$ with the prime field of $C$, as usual. {\it Suppose $c\notin\Q$.}\/ Then $R$ satisfies the hypotheses of Theorem~\ref{thm:RosenlichtOrder1}: First, 
$$\frac{1}{R}\  =\  \frac{1}{Y+1}+\frac{c}{Y}.$$
Next, if $\frac{1}{R} = a\,\frac{\partial U}{U}$, $a\in C^\times$, $U\in C(Y)^\times$,
then by [ADH, 4.1.3] and \cite[Chapter~V, Theorem~5.2]{Lang} both $1/a$ and~$c/a$ are integers, hence
$c\in\Q$, a contradiction.  
Equip~$E:=C(Y)$ with  
the valuation $v\colon E^\times\to\Z$ with $v(C^\times)=\{0\}$ and $v(Y)=1$. Then~$E$ 
is  $\d$-valued   with $\Psi_E=\{-1\}$. Since $v(1/R)=-1$, there is also no~$V\in E^\times$ with~$\partial V=1/R$. For the constant field $C$ of $K$ this yields by Theorem~\ref{thm:RosenlichtOrder1}:

\begin{cor}\label{cor:RosenlichtOrder1}
Let $c\in C\setminus \Q$. Then $P(y)=0$ for some $y$ in a differential field extension of $K$ and transcendental over $K$; for any such $y$ we have $C_{K\<y\>}=C$.
\end{cor}


\section{Revisiting [ADH, 16.0.3]}\label{sec:revisiting}

\noindent
The following important fact about closed $H$-fields stands in marked contrast to Corollary~\ref{cor:RosenlichtOrder1}. It is [ADH, 16.0.3], and was applied in \cite{ADHdim} to gain information about the zero sets of differential polynomials in closed $H$-fields:

\begin{theorem} \label{thm:16.0.3}
Every closed $H$-field is $\d$-closed in every  $H$-field extension with the same constant field.
\end{theorem}

\noindent
Any pre-$H$-field extension of an $H$-field with the same residue field
is automatically an $H$-field  with the same constant field, so in 
Theorem ~\ref{thm:16.0.3} we can replace ``$H$-field extension with the same constant field''  by
``pre-$H$-field extension with the same residue field''.
But we cannot further weaken this to ``pre-$H$-field extension with the same constant field'' as we show
in the next subsection. (This justifies a remark in~[ADH, ``Notes and comments'' on p.~684].)

\subsection*{How not to use Theorem~\ref{thm:16.0.3}}
We  work in the differential field $\T$ of transseries. Let $c\in\R^>$, and let $P=P_c\in\R\{Y\}$ be as introduced at the end
of the previous section\endnote{The differential polynomial $P_c$ studied here was inspired by an example
attributed to McGrail and Marker in \cite[Example~2.20]{HrushovskiItai}, and suggested to us by James Freitag.}. Note that $P(0)=P(-1)=0$. Let $y\in\T\setminus\{0,-1\}$ and put 
$$U(y)\ :=\ \abs{y}^{c}(y+1)\in \T^\times,$$
so
$$U(y)^\dagger\  =\   c\frac{y'}{y}+\frac{y'}{y+1}\  =\  y'\cdot \frac{c(y+1)+y}{y(y+1)}$$
and therefore 
\begin{align*}
P(y)=0 &\quad\Longleftrightarrow\quad U(y)^\dagger=1 \\ &\quad\Longleftrightarrow\quad U(y)\in \R^\times\ex^x \\
 &\quad\Longleftrightarrow\quad  \abs{y}^c(y+1) = a\ex^x \text{ for some $a\in \R^\times$.}
\end{align*}
Hence $y\succ 1$ whenever $P(y)=0$. More precisely:

\begin{lemma}\label{pyo}
Let $a\in\R^\times$, $y\in \T^\times$ with  $\abs{y}^c(y+1) = a\ex^x$. 
Then  $y \sim b \ex^{x/(c+1)}$ where $b=a^{1/(c+1)}$ if $a>0$ and $b=-(-a)^{1/(c+1)}$ if $a<0$.
\end{lemma}
\begin{proof}
Note that $\abs{y}^c y \sim \abs{y}^c(1+y) = a\ex^x$.
Hence if $a>0$, then $y>0$ and~$y\sim b\ex^{x/(c+1)}$
with~$b:=a^{1/(c+1)}$, and if $a<0$, then $y<0$ and $(-y)^{c+1}\sim -a\ex^x$,
thus~$y\sim b\ex^{x/(c+1)}$ for~$b:=-(-a)^{1/(c+1)}$.
\end{proof}

\noindent
In the next lemma we let $b\in \R^\times$, and for $z\in\T$, $z\prec 1$ we set
$$Q(z)\ :=\  (1+z)^c ( 1+\varepsilon +z)-1\qquad\text{where  $\varepsilon:=b^{-1}\ex^{-x/(c+1)}\prec 1$.}$$

\begin{lemma}
Let $y=b\ex^{x/(c+1)}(1+z)$ where $z\prec 1$. Then $P(y)=0\Longleftrightarrow Q(z)=0$.
\end{lemma}
\begin{proof}
We have $\abs{y}^c(y+1)=\abs{b}^c b\ex^x \big(Q(z)+1\big)$ and $Q(z)+1\sim 1$.
Assume $P(y)=0$, and
take~$a\in\R^\times$ such that~$\abs{y}^c(y+1)=a\ex^x$.
Then $a=\abs{b}^{c}b$ and $Q(z)=0$. Conversely, if $Q(z)=0$, then  $\abs{y}^c(y+1)=a\ex^x$ for $a:=\abs{b}^c b$ and so $P(y)=0$.
\end{proof}

\noindent
Combining Corollary~\ref{cor:hens, 2} with the previous lemma yields:

\begin{cor}\label{cor:zeros of P}
For each $b\in\R^\times$ there is a unique $y\in\T^\times$ such that $P(y)=0$ and~$y\sim b\ex^{x/(c+1)}$.
\end{cor}
 
\noindent
Let $H$ be a prime model of the theory of closed $H$-fields with small derivation, and identify $H$ with an ordered valued differential subfield of
$\T$; see~[ADH, p.~705]. The constant field $C$ of $H$ is the real closure of $\Q$ in $\R$. 
Take $c\in C^>$. Since $H$ is Liouville closed,   we have $d\in\R^>$ with~$f:=d\ex^{x/(c+1)}\in  H^>$.
Suppose also $c\notin\Q$ and 
let $b\in\R\setminus C$. Then Corollary~\ref{cor:zeros of P} gives $y\in\T^\times$ such that~${P(y)=0}$ and~$y\sim bf$.
So $H\langle y\rangle$  is a pre-$H$-subfield of $\T$ with~$y\notin H$, and  $C_{H\langle y\rangle}=C$ by Corollary~\ref{cor:RosenlichtOrder1}.
Hence~$H\langle y\rangle\supseteq H$ is an example of a proper $\d$-algebraic pre-$H$-field extension of a closed $H$-field with the same constant field, as promised after Theorem~\ref{thm:16.0.3}.
A similar argument gives an analogue of Corollary~\ref{cor:RosenlichtOrder1} in the category of $H$-fields:
 
\begin{cor}\label{cor:RosenlichtOrder1, H-fields}
Let $H\subseteq E$ be an extension of closed $H$-fields with small derivation such that $C_E\neq C$, and let $c\in C^>$. Then $P_c(y)=0$ for some $y\in E\setminus H$, and if~$c\notin\Q$, then for any such $y$ we have $C_{H\langle y\rangle}=C$.
\end{cor}
\begin{proof}
Take $f\in H^\times$ with $f^\dagger=1/(c+1)$. Since $E\equiv\mathbb T$, there is for each $b\in C_E^\times$ a unique $y\in E^\times$ with $P_c(y)=0$ and $y\sim bf$.
Taking $b\in C_E\setminus C$, this $y$ satisfies~$y\notin H$, with $C_{H\langle y\rangle}=C$ if also $c\notin \Q$, by Corollary~\ref{cor:RosenlichtOrder1}.
\end{proof}

\begin{example}
  Let $H := \T_x := \R[[[x]]]$ be the field of transseries in $x$ over  $\R$
  and let~$E := \T_u[[[x]]]$ be the field of transseries in $x$
  over the field of transseries~$\T_u := \R[[[u]]]$ in a second variable $u$
  (elements in $\T_u$ being constants for the derivation $\frac{d}{dx}$).
  We refer to~\cite{vdH:ln} for the construction of transseries\endnote{The transseries from~\cite{vdH:ln} are \textit{grid-based}, subject to a stronger restriction than the
  \textit{well-based} trans\-series from [ADH, Appendix~A].
  In both cases these trans\-series form $H$-closed fields.} over
  an arbitrary exponential ordered constant field like
  $\T_u$\endnote{The requirement that the constant field comes with
  an exponential function is natural from the transseries perspective in~\cite{vdH:ln},
  but not required from an $H$-field perspective: \textit{mutatis mutandis},
  the construction goes through for general ordered constant fields $C$,
  but the resulting $H$-field of transseries $\T$ is just no longer
  closed under exponentiation.  Nevertheless, it remains closed
  under exponential integration and exponentiation for transseries
  without constant terms.}.
  It is also shown there that both $H$ and~$E$ are Liouville closed and
   satisfy DIVP and so are $H$-closed by~\cite[Theorem~2.7]{ADHip}.
  Now by what precedes, there is a unique $y \in E \setminus H$
  with~${P_c(y) = 0}$, $y \sim u \ex^{x/(c+1)}$, and~$C_{H\langle y\rangle} = \R$.
\end{example}

\noindent
In the rest of this subsection $H$ is a closed $H$-field with small derivation and constant field $C$. In \cite[Section~5]{ADHdim} we called a definable set~$S\subseteq H^n$  {\it parametrizable by constants}\/  
if there are  a semialgebraic set $X\subseteq C^m$, for some~$m$, and a definable bijection $X\to S$.
Lemma~\ref{pyo}, Corollary~\ref{cor:zeros of P}, and $H\equiv \T$
 yield:

\begin{cor}\label{cor:par by const}
For every~${c\in C^>}$ the definable set $\Zero(P_c)\subseteq H$ is parametrizable by constants.
\end{cor}

\noindent
In \cite[Section~5]{ADHdim} an irreducible differential polynomial $Q\in H\{Y\}^{\ne}$ was said to {\it create a constant}\/ if
for some element $y$ in a differential field extension of $H$ with minimal annihilator $Q$ over $H$ we have $C_{H\langle y\rangle}\ne C$. We showed that for such~$Q$, if~${\order(Q)=1}$, then $\Zero(Q)$ is parametrizable by constants \cite[Pro\-po\-si\-tion~5.4]{ADHdim}. If~${c\notin\Q}$ in Corollary~\ref{cor:par by const}, then the  irreducible differential polynomial~${P_c\in H\{Y\}}$ of order~$1$   is parametrizable by constants, yet does not create a constant.
Nonetheless, it does create a constant in the following more liberal sense.

\begin{prop}
  Let $E$ be a closed $H$-field extension of $H$ and let $y \in E \setminus H$ be d-algebraic over $H$.
  Then some $c\in C_E \setminus C$ is definable in $E$ over $H\cup\{y\}$.
\end{prop}

\begin{proof}
 The definable closure $K$ of $H \cup \{y\}$ in $E$ is a differential subfield, even an
  $H$-subfield of $E$:  if $u\in K$ and $u \asymp 1$, then $u\sim c$ for a unique $c\in C_E$,
 so  $c \in K$ for this $c$. 
It remains to note that the constant field of $K$ is strictly larger than that of~$H$, by Theorem~\ref{thm:16.0.3} and
  our assumption that $y \in K \setminus H$ is d-algebraic over~$H$.
\end{proof}

\subsection*{Generalizing Theorem~\ref{thm:16.0.3}}  
In this subsection $K$ is a $\d$-valued field of $H$-type with algebraically closed constant field $C$ and divisible group $K^\dagger$ of logarithmic derivatives.
We use spectral extensions to prove an analogue of Theorem~\ref{thm:16.0.3}: 

\begin{theorem}\label{noextension} Suppose $K$ is  $\upo$-free and newtonian. Then $K$ is $\d$-closed in each 
$\d$-valued field extension $L$ of $H$-type with $C_L=C$ and $L^\dagger\cap K=K^\dagger$. 
\end{theorem} 

\noindent
This yields a generalization of Theorem~\ref{thm:16.0.3}:

\begin{cor}\label{cor:noextension}
Let $H$ be an $\upo$-free newtonian real closed  $H$-field and~$E$ be an $H$-field extension of $H$. 
Then $H$ is $\d$-closed in
$E$ iff $C_E=C_H$ and $E^\dagger\cap H=H^\dagger$.
\end{cor}
\begin{proof}
The ``only if'' direction holds by \eqref{eq:1-d-closed}. Suppose $C_E=C_H$ and $E^\dagger\cap H=H^\dagger$; we need to show that
$H$ is $\d$-closed in $E$.  
Replacing $E$ by its real closure and using
Lemma~\ref{lem:Kdagger alg closure} and Corollary~\ref{cor:dc base field ext},
it suffices to show that the $\d$-valued field $K:=H[\imag]$ is $\d$-closed in its  $\d$-valued field extension $L:=E[\imag]$.
By~[ADH, 11.7.23, 14.5.7], $K$ is $\upo$-free and newtonian. Also~${C_L=C_E[\imag]=C_H[\imag]=C}$. 
Moreover, by Lemma~\ref{lem:logder real closed}, 
$$L^\dagger\cap K\ =\ (E^\dagger\cap H)\oplus\I(H)\imag\ =\  H^\dagger\oplus\I(H)\imag\ =\ K^\dagger.$$
Now use Theorem~\ref{noextension}.
\end{proof}

\noindent
In the same way that [ADH,   16.0.3] follows from [ADH,   16.1.1], Theorem~\ref{noextension} follows from an  analogue of [ADH, 16.1.1]:

\begin{lemma}\label{lem:ADH 16.1.1, Calinfi}
Let $K$ be an $\upo$-free newtonian $\d$-valued field, $L$ a $\d$-valued field extension of $K$ of $H$-type with
$C_L=C$ and $L^\dagger\cap K=K^\dagger$, and let $f\in L\setminus K$. Suppose there is no~$y\in K\langle f\rangle\setminus K$ such that $K\langle y\rangle$ is an immediate extension of~$K$.
Then the $\Q$-linear space~$\Q\Gamma_{K\langle f\rangle}/\Gamma$ is infinite-dimensional.
\end{lemma} 

\noindent
The proof of Lemma~\ref{lem:ADH 16.1.1, Calinfi} is much like that of [ADH,   16.1.1], except where
the latter uses that any $b$ in a Liouville closed $H$-field equals $a^\dagger$ for some nonzero $a$ in that field. This might not work with elements of $K$, and the remedy is to take instead for every $b\in K$ an element $a$ in $\Univ^\times$ with $b=a^\dagger$.
The relevant computation should then take place in the differential fraction field $\Omega_L$ of $\Univ_L$ instead of in $L$ where $\Omega_L$ is equipped with a spectral extension of the valuation of $L$. For all this  to make sense, we first take an active $\phi$ in $K$ and
replace $K$ and $L$ by $K^\phi$ and $L^\phi$, arranging in this way that the derivation of $L$ (and of $K$) is small. Next we
replace~$L$ by its algebraic closure, so that $L^\dagger$ is divisible, while preserving $L^\dagger\cap K=K^\dagger$ by 
\cite[Lemma~1.2.1]{ADH4},  and also preserving the other conditions on $L$ in Lemma~\ref{lem:ADH 16.1.1, Calinfi}, as well as the derivation of $L$ being small.  This allows us to
identify  $\Univ$ with a differential subring of $\Univ_L$
 and accordingly $\Omega$ with a differential subfield of $\Omega_L$.  We equip~$\Omega_L$ with a spectral extension of the valuation of $L$ 
and make $\Omega$ a valued subfield of $\Omega_L$. Then the valuation of $\Omega$ is a spectral extension of the valuation of $K$ to $\Omega$, so we have
the following inclusions of $\d$-valued fields:
$$\xymatrix{L \ar[r] & \Omega_L \\
K \ar[u] \ar[r] & \Omega \ar[u]}$$
With these preparations we can now give  the proof of Lemma~\ref{lem:ADH 16.1.1, Calinfi}:

\begin{proof} As we just indicated we arrange that $L$ is algebraically closed with small derivation,  and with an inclusion diagram of $\d$-valued fields involving $\Omega$ and $\Omega_L$, as above. (This will not be used until we arrive at the Claim below.)

By [ADH, 14.0.2], $K$ is asymptotically $\d$-algebraically maximal. Using this and the assumption about $K\langle f\rangle$ it follows
as in the proof of [ADH, 16.1.1] that
there is no divergent pc-sequence in $K$
with a pseudolimit in~$K\langle f\rangle$. Thus every~$y$ in~$K\langle f \rangle \setminus K$ has a 
 {\it a  best approximation in $K$,}\/ that is, an element $b\in K$ such that  $v(y-b)=\max v(y-K)$.
For such $b$ we have $v(y-b)\notin \Gamma$, since $C_L=C$.

Now pick a best approximation $b_0$ in $K$ to $f_0:=f$, and set
$f_1:= (f_0-b_0)^\dagger$. Then~$f_1\in K\langle f\rangle \setminus K$, since $L^\dagger\cap K=K^\dagger$ and $C=C_L$. Thus
$f_1$ has a best approximation $b_1$ in $K$, and continuing this way, we
obtain a sequence $(f_n)$ in~${K\langle f \rangle\setminus K}$ and a sequence
$(b_n)$ in $K$, such that $b_n$ is a best approximation in $K$ to~$f_n$ and~$f_{n+1}=(f_n-b_n)^\dagger$ for all $n$. Thus
$v(f_n-b_n)\in \Gamma_{K\langle f \rangle}\setminus \Gamma$ for all $n$. 

\claim{$v(f_0-b_0),v(f_1-b_1), v(f_2-b_2),\dots$ are $\Q$-linearly independent over $\Gamma$.}

\noindent
To prove this claim, take $a_n\in \Univ^\times$ with $a_n^\dagger=b_n$ for
$n\ge 1$. Then in $\Omega_L$,
$$f_n-b_n\ =\ (f_{n-1}-b_{n-1})^\dagger-a_n^\dagger\ =\ \left(\frac{f_{n-1}-b_{n-1}}{a_n}\right)^\dagger \qquad (n\ge 1).$$
With $\psi:= \psi_{\Omega_L}$ and $\alpha_n=v(a_n)\in \Gamma_{\Omega}\subseteq \Gamma_{\Omega_L}$ for $n\ge 1$, we get 
\begin{align*} v(f_n-b_n)\ &=\ \psi\big(v(f_{n-1}-b_{n-1})-\alpha_n\big),\ 
 \text{ so by an easy induction on $n$,}\\
 v(f_n-b_n)\ &=\ \psi_{\alpha_1,\dots, \alpha_n}\big(v(f_0-b_0)\big)   \qquad (n\ge 1).
 \end{align*} 
 (The definition of the functions $\psi_{\alpha_1,\dots, \alpha_n}$ is given just before [ADH, 9.9.2].)
Suppose towards a contradiction that $v(f_0-b_0), \dots, v(f_n-b_n)$ are $\Q$-linearly dependent
over~$\Gamma$. Then we have $m<n$ and $q_1,\dots, q_{n-m}\in \Q$ such that
$$v(f_m-b_m) + q_1v(f_{m+1}-b_{m+1}) + \cdots + q_{n-m}v(f_n-b_n)\in \Gamma.$$ 
For $\gamma:= v(f_m-b_m)\in \Gamma_L\setminus\Gamma$ this gives
$$\gamma + q_1\psi_{\alpha_{m+1}}(\gamma) + \cdots + q_{n-m}\psi_{\alpha_{m+1},\dots,\alpha_n}(\gamma)\in \Gamma.$$
By [ADH, 14.2.5] we have $\I(K)\subseteq K^\dagger$,  so the $H$-asymptotic couple of $\Omega$ is closed  with $\Psi_\Omega\subseteq\Gamma$, by Lemma~\ref{lem:as cpl Omega}. 
 Hence $\gamma\in\Gamma_\Omega$ by [ADH, 9.9.2].
 Together with~$\Psi_\Omega\subseteq\Gamma$ and~$\alpha_{m+1},\dots,\alpha_n\in\Gamma_\Omega$ this gives 
$\psi_{\alpha_{m+1}}(\gamma),\dots, \psi_{\alpha_{m+1},\dots,\alpha_n}(\gamma)\in\Gamma$
and thus $\gamma\in\Gamma$, a contradiction.
\end{proof}

\noindent
We augment the language $\mathcal L_{\preceq}=\{0,1,{-},{+},{\,\cdot\,},{\der},{\preceq}\}$ of   valued differential rings by a unary relation symbol to obtain the language $\mathcal L_{c}$. We construe each valued
differential field  as an $\mathcal L_c$-structure by interpreting the new relation symbol as its constant field. 
In \cite[Proposition~6.2]{ADHdim} we showed that if $H$ is a closed $H$-field, then a nonempty definable set $S\subseteq H^n$ is
co-analyzable (relative to $C_H$) iff $\dim S=0$.
(See Section~6 of \cite{ADHdim} for the definition of ``co-analyzable''\endnote{This terminology, coming from \cite{HHM}, is slightly unfortunate
in our context in light of the important role played by \'Ecalle's analyzable functions~\cite{Ecalle} in the theory of transseries.}.)
This used Theorem~\ref{thm:16.0.3} and a model-theoretic test for co-analyzability from \cite{HHM} (cf.~\cite[Proposition~6.1]{ADHdim}).
From Corollary~\ref{cor:noextension} we obtain a partial generalization of this fact:

\begin{cor}\label{cor:coanal}
Let  $H$ be an $\upo$-free newtonian real closed $H$-field. If  $S\subseteq H^n$ and~$\dim S=0$, then $S$ is co-analyzable. \end{cor}
\begin{proof}
If $E$ is an elementary extension of $H$, then $E^\dagger\cap H=H^\dagger$. Hence
for each~${P\in H\{Y\}^{\ne}}$, the zero set $\Zero(P)\subseteq H$ 
is co-analyzable by Corollary~\ref{cor:noextension}
and by  \cite[Pro\-po\-si\-tion~6.1]{ADHdim} applied to the $\mathcal L_{c,A}$-theory $\operatorname{Th}(H_A)$ where $A$ is  the finite set of nonzero coefficients of $P$. This special case implies the general case, since
for each~${S\subseteq H^n}$ with $\dim S=0$ there are $P_1,\dots,P_n\in H\{Y\}^{\ne}$  such that  
 $S$ is contained in  $\Zero(P_1)\times\cdots\times\Zero(P_n)$. 
\end{proof}

\noindent
Likewise, using
Theorem~\ref{noextension} instead of Corollary~\ref{cor:noextension} yields:

\begin{cor}\label{kupon}
If $K$ is $\upo$-free and newtonian, then
each set~$S\subseteq K^n$ such that~$\dim S=0$   is co-analyzable.
\end{cor}

\noindent
Thus for $K$, $S$  as in Corollary~\ref{kupon} and countable $C$, all $S\subseteq K^n$
of dimension $0$ are countable, by \cite[Pro\-po\-si\-tion~6.1]{ADHdim}. 
Note that Corollary~\ref{cor:coanal} applies to the valued differential field $\T_{\log}$ of logarithmic transseries
from [ADH, Appendix~A]\endnote{Allen Gehret has   a different proof that the zero set of each nonzero univariate differential polynomial
over $\T_{\log}$ is co-analyzable.}.

\section{Relative Differential Closure in $H$-fields}\label{sec:diff cl H-fields}

\noindent
In this section we turn to relative differential closure in the $H$-field setting, and
we use Theorem~\ref{thm:16.0.3} to relate $\d$-closedness to the elementary substructure property
and $\d$-closure to Newton-Liouville closure.

Let $\mathcal L_\der=\{0,1,{-},{+},{\,\cdot\,},{\der}\}$  be the language of differential rings, 
a sublanguage of the language $\mathcal L:=\mathcal{L}_\der \cup\{\le, \preceq\}$ of ordered valued differential rings (see [ADH, p.~678]). 
 In  this section we  also let~$M$ be an $H$-closed field  and
$H$ a pre-$H$-subfield of~$M$ whose valuation ring and constant field we denote by $\mathcal{O}$ and $C$.
Construing~$H$ and~$M$ as $\mathcal L$-structures in the usual way,
$H$ is an $\mathcal L$-substructure of~$M$.
We also use the sublanguage~$\mathcal L_{\preceq}:=\mathcal L_\der\cup\{ {\preceq} \}$ of~$\mathcal L$, so $\mathcal L_{\preceq}$ is the language of valued differential rings. We
expand the $\mathcal L_\der$-structure~$H[\imag]$ to an   $\mathcal L_{\preceq}$-structure by interpreting~$\preceq$ 
as the dominance relation associated to the valuation ring $\mathcal O+\mathcal O\imag$ of $H[\imag]$;
we expand likewise~$M[\imag]$ to an   $\mathcal L_{\preceq}$-structure by interpreting~$\preceq$ 
as the dominance relation associated to the valuation ring $\mathcal O_{M[\imag]}=\mathcal O_M+\mathcal O_M\imag$ of~$M[\imag]$. 
Then $H[\imag]$ is an $\mathcal L_{\preceq}$-substructure of~$M[\imag]$.
By~$H\preceq_{\mathcal{L}} M$ we mean that $H$ is an elementary $\mathcal L$-substructure of $M$, and we use expressions like ``$H[\imag]\preceq_{\mathcal{L}_{\preceq}} M[\imag]$''  in the same way; of course, the two
uses of the symbol $\preceq$ in the latter are unrelated. 

By 
Corollary~\ref{cor:dc base field ext}, $H$ is $\d$-closed in~$M$ iff $H[\imag]$ is $\d$-closed in $M[\imag]$.

\begin{lemma}
Suppose $M$ has small derivation. Then 
$$H\ \preceq_{\mathcal L_\der}\ M\ \Longleftrightarrow\ H[\imag]\ \preceq_{\mathcal L_\der}\ M[\imag].$$
Also, if $H\preceq_{\mathcal L_\der} M$, then $H\preceq_{\mathcal L} M$ and  $H[\imag]\preceq_{\mathcal{L}_{\preceq}} M[\imag]$.
\end{lemma}
\begin{proof}
The  forward direction in the equivalence is obvious. For the converse, let~$H[\imag]\preceq_{\mathcal L_\der}M[\imag]$.
We have  $M\equiv_{\mathcal{L}_{\der}}\mathbb T$ by [ADH, 16.6.3].
Then  [ADH, 10.7.10] yields an $\mathcal L_\der$-formula defining $M$ in $M[\imag]$,  so the same formula defines $M\cap H[\imag]=H$ in~$H[\imag]$, and thus $H\preceq_{\mathcal L_\der} M$. For the ``also'' part, use that the squares of $M$ are the nonnegative elements in its ordering,  that $\mathcal O_M$ is then definable as the convex hull of $C_M$ in $M$ with respect to this ordering, and
if $H\preceq_{\mathcal L_\der} M$, then each $\mathcal L_\der$-formula defining $\mathcal O_M$ in $M$ also defines $\mathcal O=\mathcal O_M\cap H$ in $H$.
\end{proof}

\noindent
By [ADH, 14.5.7, 14.5.3], if $H$ is $H$-closed, then $H[\imag]$ is weakly $\d$-closed.
The next proposition complements Theorem~\ref{thm:16.0.3} and~[ADH,   16.2.5]:  

\begin{prop}\label{prop:16.0.3 converse}
The following are equivalent:
\begin{enumerate}
\item[\textup{(i)}] $H$ is $\d$-closed in $M$;
\item[\textup{(ii)}]  $C=C_M$ and $H\preceq_{\mathcal L} M$;
\item[\textup{(iii)}] $C=C_M$ and $H$ is $H$-closed.
\end{enumerate}
\end{prop}

\noindent
First a lemma, where {\em extension\/} refers to an extension of valued differential fields, and where~$r\in\N$. 

\begin{lemma}\label{lem:r-newtonian descends} 
Let $K$ be a  $\upl$-free $H$-asymptotic field which is $r$-$\d$-closed in an $r$-newtonian ungrounded $H$-asymptotic extension $L$. Then $K$ is also $r$-newtonian.
\end{lemma}
\begin{proof}
Let $P\in K\{Y\}^{\neq}$ be quasilinear of order~$\leq r$. Then $P$ remains quasilinear when viewed as differential polynomial over $L$, by 
\cite[Lemma~1.7.9]{ADH4}.  Hence $P$ has a zero~$y\preceq 1$ in $L$, which lies in $K$ since $K$ is $r$-$\d$-closed in $L$.
\end{proof}

\begin{proof}[Proof of Proposition~\ref{prop:16.0.3 converse}] Assume (i).  Then $C=C_M$ and $H$ is a Liouville closed $H$-field, by \eqref{eq:0-d-closed}, \eqref{eq:weakly 1-d-closed}, and \eqref{eq:1-d-closed}.
We have $\omega(M)\cap H=\omega(H)$ since $H$ is weakly $1$-$\d$-closed in~$M$, and
$\sigma\big(\Upg(M)\big)\cap H=\sigma\big(\Upg(M)\cap H\big)=\sigma\big(\Upg(H)\big)$ since~$H$ is $2$-$\d$-closed in $M$ and~$\Upg(M)\cap H=\Upg(H)$ by~[ADH, p.~520]. Now~$M$ is
Schwarz closed~[ADH, 14.2.20], so $M=\omega(M)\cup\sigma\big(\Upg(M)\big)$, hence also~$H=\omega(H)\cup\sigma\big(\Upg(H)\big)$, thus 
$H$ is  
Schwarz closed [ADH, 11.8.33]; in particular, $H$ is $\upo$-free.
By Lemma~\ref{lem:r-newtonian descends}, $H$ is newtonian.
This shows (i)~$\Rightarrow$~(iii). The implication  (iii)~$\Rightarrow$~(i) is Theorem~\ref{thm:16.0.3}, and~(iii)~$\Leftrightarrow$~(ii) follows from [ADH, 16.2.5].
\end{proof}

\noindent
Next a consequence of [ADH, 16.2.1], but note first that $H(C_M)$ is an $H$-subfield of $M$
and $\d$-algebraic over $H$, and  recall that each $\upo$-free $H$-field
has a Newton-Liouville closure,   as defined in~[ADH, p.~669].
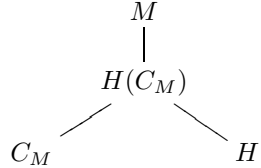
\[
\xymatrix@R-1pc@C-1pc{
& M \ar@{-}[d] & \\
& H(C_M) & \\
C_M \ar@{-}[ur] & & H \ar@{-}[ul]
}
\]
\begin{cor}\label{cor:nlc, 1} 
If $H$ is  $\upo$-free, then the differential closure of~$H$ in $M$ is a New\-ton-Liou\-ville closure of the $\upo$-free $H$-subfield $H(C_M)$ of $M$. 
\end{cor}

\noindent
For (pre-)$\HLO$-fields, see [ADH, 16.3], and for the Newton-Liouville closure of a pre-$\HLO$-field see~[ADH, 16.4.8].
Let $\boldsymbol M$ be the expansion of $M$ to a $\HLO$-field, and  let   $\boldsymbol H$, $\boldsymbol H(C_M)$ be the   expansions of  $H$, $H(C_M)$, respectively, to   pre-$\HLO$-subfields of $\boldsymbol M$;  then   $\boldsymbol H(C_M)$ is a $\HLO$-field.
By   Pro\-po\-si\-tion~\ref{prop:16.0.3 converse}, the  $\d$-closure~$H^{\operatorname{da}}$ of $H$ in~$M$ is $H$-closed and hence has a unique expansion  $\boldsymbol H^{\operatorname{da}}$  
to a $\HLO$-field. 
Then~$\boldsymbol H\subseteq \boldsymbol H(C_M)\subseteq \boldsymbol H^{\operatorname{da}}\subseteq\boldsymbol M$.
 
\begin{cor}\label{cor:nlc, 2}  
The $\HLO$-field $\boldsymbol H^{\operatorname{da}}$ is a New\-ton-Liouville closure of~$\boldsymbol H(C_M)$.
\end{cor}
\begin{proof}
Let $\boldsymbol H(C_M)^{\operatorname{nl}}$ be a  Newton-Liouville closure of $\boldsymbol H(C_M)$.
Since $\boldsymbol H^{\operatorname{da}}$   is $H$-closed and extends $\boldsymbol H(C_M)$,  there is an embedding $\boldsymbol H(C_M)^{\operatorname{nl}}\to \boldsymbol H^{\operatorname{da}}$ over $\boldsymbol H(C_M)$, and any such embedding is an isomorphism, thanks to  Theorem~\ref{thm:16.0.3}.
\end{proof}

\section{Relative Differential Closure in Hardy Fields}\label{sec:d-cl hardy}

\noindent
Specializing now to Hardy fields, assume in this section that
$H$ is a Hardy field and set~$K:=H[\imag]\subseteq \Calinf[\imag]$, an $H$-asymptotic extension of $H$. 
By definition, $H$ is $\d$-maximal iff~$H$ is $\d$-closed in every Hardy field extension of $H$. Moreover:

\begin{cor}\label{cor:char d-max}
The following are equivalent:
\begin{enumerate}
\item[\textup{(i)}] $H\supseteq \R$ and $H$ is $H$-closed;
\item[\textup{(ii)}] $H$ is $\d$-maximal;
\item[\textup{(iii)}] $H$ is $\d$-closed in some  $\d$-maximal Hardy field extension of~$H$.
\end{enumerate}
If one of these conditions holds, then $K$ is weakly $\d$-closed.
\end{cor}
\begin{proof}
The equivalence (i)~$\Leftrightarrow$~(ii) is part of Theorem~\ref{thm:maxhardymainthm}, (ii)~$\Rightarrow$~(iii) is trivial, and (iii)~$\Rightarrow$~(i) follows from the implication (i)~$\Rightarrow$~(iii) of
Proposition~\ref{prop:16.0.3 converse}. For the rest, use the remark before Proposition~\ref{prop:16.0.3 converse}.
\end{proof}

\noindent
Thus by Corollary~\ref{cor:dc base field ext}:

\begin{cor}\label{cor:d-max weakly d-closed}
If $H$ is $\d$-maximal and~$E$ is a Hardy field extension of~$H$, then
$K$ is $\d$-closed in $E[\imag]$.
\end{cor}

\noindent
Using Theorem~\ref{noextension} we can strengthen  Corollary~\ref{cor:d-max weakly d-closed}:

\begin{cor}\label{dmaxdom} Suppose $H$ is $\d$-maximal and $L\supseteq K$ is a differential subfield of~$\Calinf[\imag]$  such that $L$ is a $\d$-valued $H$-asymptotic extension of $K$ with respect to some dominance relation on $L$. 
Then $K$ is $\d$-closed in $L$. 
\end{cor}
\begin{proof}  The $\d$-valued field $K$ is $\upo$-free and newtonian by [ADH, 11.7.23, 14.5.7]. Also $L^\dagger\cap K=K^\dagger$ by Corollary~\ref{cor:fexphii}. Now apply Theorem~\ref{noextension}.
\end{proof} 

\noindent
We do not require  the dominance relation on $L$ in Corollary~\ref{dmaxdom} to be the restriction to $L$ of the relation $\preceq$ on $\c[\imag]$.  

A differential subfield $L$ of $\Calinf[\imag]$ is said to {\it come from a Hardy field}\/ if~${L=E[\imag]}$ for some Hardy field $E$. 
By \cite[Section~2]{ADH5} this is equivalent to: $\imag \in L$ and $\Re f,\Im f\in L$ for all $f\in L$. Corollary~\ref{dmaxdom} for $L$ coming from a Hardy field
falls under Corollary~\ref{cor:d-max weakly d-closed}. However, not every differential subfield of $\Calinf[\imag]$ containing~$\C$ comes from a Hardy field:
 \cite[Section~5]{ADH5} has
an example of a differential subfield~${L\supseteq \C}$ of $\Gom[\imag]$ not coming from a Hardy field, yet  the relation  $\preceq$ on $\c[\imag]$    restricts to a dominance relation on $L$
making~$L$ a $\d$-valued field of $H$-type. In the next example (not used later) we employ a variant of this construction 
to obtain~$H$,~$K$,~$L$ as in Corollary~\ref{dmaxdom} where~$L$
equipped with the restriction of $\preceq$ is
a $\d$-valued field of $H$-type, but  $L$ doesn't come from a Hardy field:

\begin{example} 
Let $M$ be a maximal analytic Hardy field. Then~$M$ is $H$-closed by Corollary~\ref{thm:extend to H-closed},
the $\d$-closure $H$  of $\R$ in $M$ is a $\d$-maximal  Hardy field by Corollary~\ref{cor:char d-max}, and no~$h\in H$ is transexponential
 by \cite[Lem\-ma~5.1]{ADH5}.
Theo\-rem~1.3 in~\cite{Boshernitzan86} and \cite[Corollary~5.24]{ADH5}) yield 
a transexponential~$z\in M$. Take any $h\in\R(x)$ with $0\ne h\prec 1$, and put~$y:=z(1+h\ex^{x\imag})\in\Gom[\imag]$.
Now $z$ is $H$-hardian, 
so by~\cite[Lemma~5.16]{ADH5}, $y$ generates a differential subfield~$L_0:=H\langle y\rangle$
of $\Gom[\imag]$, and~$\preceq$ restricts to a dominance relation on $L_0$ which makes $L_0$ a $\d$-valued field
of $H$-type with constant field~$\R$. Then~$L:=L_0[\imag]$ is a differential subfield of $\Gom[\imag]$ with
constant field $\C$, and~$\preceq$ on~$\c[\imag]$ restricts to a dominance relation on~$L$ that makes $L$ a $\d$-valued field
of $H$-type extending $K$ as a $\d$-valued field and which  does not come from a Hardy field, since $\Im y=zh\sin x$ oscillates. 
By Corollary~\ref{dmaxdom}, $K$ is $\d$-closed in $L=K\langle y\rangle$.
\end{example}

\noindent
The following is  Theorem~B from the introduction:

\begin{cor}\label{corhe}
Suppose $H\supset\R$ has \textup{DIVP}, and $E\supseteq H$ is a Hardy field. Then 
$$H \text{ is $\d$-closed in  $E$}\ \Longleftrightarrow\  E\cap \exp(H) \subseteq H.$$
\end{cor}
\begin{proof}
By \cite[Corollary~2.6]{ADHip}, $H$ is $\upo$-free and newtonian. In particular, $H$ is closed under integration, by [ADH, 14.2.2]. Applying \textup{DIVP} to ordinary polynomials, we see that $H$ is also real closed. 
It is clear that 
if  $H$ is $\d$-closed in~$E$, then~${E\cap \exp(H) \subseteq H}$. Conversely,
suppose that $E\cap\exp(H)\subseteq H$. If $y\in E^>$ with~${y^\dagger\in H}$, then~${\log y\in H}$ and so $y \in E\cap\exp(H)\subseteq H$. This yields~$E^\dagger\cap H=H^\dagger$. Hence~$H$ is $\d$-closed in~$E$ by Corollary~\ref{cor:noextension}. \end{proof}

\begin{remarks}
Suppose $H\supset\R$  has DIVP.  Then $H$ is  closed under integration and real closed by the proof of Corollary~\ref{corhe}, so $H$  is Liouville closed iff~${\exp(H)\subseteq H}$. Thus by  \cite[Corollary~2.6]{ADHip},
if $\exp(H)\subseteq H$, then $H$ is $H$-closed.
For an example of an analytic $H\supset\R$ with  DIVP that is not Liouville closed, consider the specialization~$F$
of $\T$ with respect to the convex subgroup $$\Delta\ :=\ \big\{ \gamma\in \Gamma_{\T}:\ \text{$\abs{\gamma}\le n\,v(x^{-1})$ for some $n$}\big\}$$
of its value group $\Gamma_{\T}=v(\T^\times)$. Then $F$ is an $H$-field with constant field $\R$ and value group $\Delta$, and $F$ has DIVP but is not
Liouville closed:
see \cite[Lemma~14.3 and subsequent remark]{hf2}. 
By [ADH, 14.1.2, 15.0.2], the $\Delta$-coarsening of the valued differential field $\T$ is $\d$-henselian,
hence [ADH, 7.1.3] yields an embedding $F\to\mathbb T$ of differential fields that is the identity on $\R$, and $F$ and $\T$ being real closed, this is even an embedding of
ordered valued differential fields. Now $\T$ is isomorphic over $\R$ to an analytic Hardy field containing $\R$ by \cite[Corollary~7.9]{AvdD}. Hence $F$ is isomorphic to an analytic Hardy field $H\supset \R$ with DIVP that is not Liouville closed. 
\end{remarks} 

\noindent
Recall from Section~\ref{sec:prelims} that the $\d$-perfect hull $\operatorname{D}(H)$ of $H$   is defined as the intersection of all $\d$-maximal Hardy field extensions of $H$.
By the next result we only need to consider  here  $\d$-algebraic Hardy field extensions of $H$:

\begin{cor}
$$\operatorname{D}(H)\ =\ \bigcap\big\{ M: \text{$M$ is a $\d$-maximal $\d$-algebraic Hardy field extension of $H$} \big\}.$$
\end{cor}
\begin{proof}
We only  need to show the inclusion~``$\supseteq$''. So let $f$ be an element of every $\d$-maximal $\d$-algebraic Hardy field extension of $H$, and let $M$ be any $\d$-maximal Hardy field extension of $H$; we need to show $f\in M$. Let $E$ be the $\d$-closure of~$H$ in $M$.
Then~$E$ is $\d$-algebraic over~$H$, and by Corollary~\ref{cor:char d-max}, $E$ is $\d$-maximal. Hence~$f\in E$, and thus~$f\in M$ as required.
\end{proof}

\noindent
We can now also prove a variant of Lemma~\ref{lem:Dx Ex}   for $\Ginf$- and $\Gom$-Hardy fields:

\begin{cor}\label{cor:D(H) smooth}
Suppose $H$ is a $\Ginf$-Hardy field. Then
\begin{align*}
\operatorname{D}(H)\ &=\ \bigcap\big\{ M: \text{$M\supseteq H$ is a $\d$-maximal $\Ginf$-Hardy field} \big\} \\
&=\ \big\{ f\in \operatorname{E}^\infty(H): \text{$f$ is $\d$-algebraic over $H$}\big\}.
\end{align*}
Likewise with $\omega$ in place of $\infty$.
\end{cor}
\begin{proof}
With both equalities replaced by  ``$\subseteq$'', this follows from the definitions and the remarks
following Proposition~\ref{prop:Hardy field ext smooth}.
Let $f\in \operatorname{E}^\infty(H)$ be $\d$-algebraic over $H$; we claim that $f\in\operatorname{D}(H)$.
To prove this claim, let $E$ be a $\d$-maximal Hardy field extension of~$H$; it is enough to show that then $f\in E$.
Now $F:=E\cap\Ginf$ is a $\Ginf$-Hardy field extension of~$H$ which is $\d$-closed in $E$, by Proposition~\ref{prop:Hardy field ext smooth}, and hence $\d$-maximal by Corollary~\ref{cor:char d-max}. Thus we may replace $E$ by $F$ to arrange that $E\subseteq\Ginf$, and then take a $\Ginf$-maximal Hardy field
extension $M$ of $E$. Now~$f\in \operatorname{E}^\infty(H)$ gives~$f\in M$, and $E$ being $\d$-maximal and $f$ being $\d$-algebraic over $E$ yields $f\in E$. The proof for~$\omega$ in place of $\infty$ is similar.
\end{proof}

\noindent
We say that $H$ is {\it bounded}\/ if there is a germ $\phi\in\c$ such that $h\le \phi$ for all $h\in H$.
No maximal Hardy field, no maximal smooth Hardy field, and no maximal analytic Hardy field
is bounded. However,
if $H$ has countable cofinality (as ordered set), then $H$ is bounded. (See \cite[Section~5]{ADH5}.)
If $H$ is bounded, then $\Ex(H)$ is $\d$-algebraic over $H$ (see \cite[Theorem~5.20]{ADH5}) and hence $\Dx(H)=\Ex(H)$ by Lemma~\ref{lem:Dx Ex}. 
If in addition $H\subseteq\c^\infty$, then $\Ex^\infty(H)$ is also $\d$-algebraic over $H$, and likewise
with $\omega$ in place of $\infty$  \cite[Theorem~5.20]{ADH5}. 
Combined  with Corollary~\ref{cor:D(H) smooth}, this yields:

\begin{cor}\label{cor:Bosh Conj 1}
If $H\subseteq\Ginf$ is  bounded, then $\Dx(H)=\Ex(H)=\Ex^\infty(H)$. Likewise with $\omega$ in place of $\infty$.
\end{cor}

\noindent
Let   $\Ex:=\Ex(\Q)$ be the perfect hull of the Hardy field $\Q$. From  Corollary~\ref{cor:Bosh Conj 1} we obtain the next result, which 
establishes Theorem~A from the introduction:

\begin{cor}\label{cor:Bosh Conj 1 for H=Q}
 $\Ex\ =\ \Ex^\infty(\Q)\ =\ \Ex^{\omega}(\Q)\ =\ \Dx(\Q)$.
\end{cor}

\begin{question}
Do the following implications hold for all $H$?
$$H\subseteq\Ginf\ \Longrightarrow\ \Ex(H)\subseteq\Ex^\infty(H), \qquad H\subseteq\Gom\ \Longrightarrow\ \Ex(H)\subseteq\Ex^\infty(H)\subseteq\Ex^\omega(H).$$
These implications hold if $\Dx(H)=\Ex(H)$, but we don't know whether $\Dx(H)=\Ex(H)$ for all $H$; see also~\cite[p.~144]{Boshernitzan82}. 
\end{question}

 \noindent
By Theorem~\ref{thm:Bosh order 1} below, each $\d$-perfect Hardy field is $1$-$\d$-closed in all its Hardy field extensions.
 Here, $Y$ and $Z$ are distinct indeterminates.  

\begin{theorem}\label{thm:Bosh order 1}
Let  $P\in H[Y,Z]^{\neq}$, and suppose  $y\in\c^1$ lies in a Hausdorff field extension of $H$ and $P(y, y')=0$. Then $y\in \Dx(H)$. 
\end{theorem}

\noindent
This is stated in \cite[Theorem~11.8]{Boshernitzan82}, where the proof is only indicated; for a
detailed proof, see \cite[Theorem~6.3.14]{ADHmax}.
Every $\d$-maximal Hardy field is $1$-newtonian (Theorem~\ref{thm:maxhardymainthm} or \cite[Lemma~11.12]{ADH2}).
Together with   Lemma~\ref{lem:r-newtonian descends}, this yields:

\begin{cor}\label{cor:d-perfect => 1-newt}
Every $\d$-perfect Hardy field  is $1$-newtonian.
\end{cor}

\noindent
By Theorem~\ref{thm:Bosh order 1} and Corollary~\ref{cor:d-perfect => 1-newt},  $\Ex$ is $1$-$\d$-closed in all its Hardy field extensions and $1$-newtonian.
However, $\Ex$ is not $2$-linearly surjective by \cite[Proposition~3.7]{Boshernitzan87}, 
so $\Ex$ is not weakly $2$-$\d$-closed in any  $\d$-maximal Hardy field extension of $\Ex$ (see Lemma~\ref{lem:weakly r-d-closed})
and   $\Ex$ is not $2$-linearly newtonian (see [ADH, 14.2.2]).

\medskip
\noindent
The material at the end of Section~\ref{sec:diff cl H-fields} has consequences
for the relationship between Newton-Liouville closure and $\d$-closure  in the context of Hardy fields.
For this, we recall from \cite[Section~12]{ADH2}  that every $\d$-maximal Hardy field $M$ has a unique expansion
to a $\HLO$-field $\boldsymbol M$, and that every $H$ has a unique expansion to a pre-$\HLO$-field~$\boldsymbol H$ such that $\boldsymbol H\subseteq\boldsymbol M$ for all $\d$-maximal Hardy fields $M\supseteq H$; we call
$\boldsymbol H$ the {\it canonical $\HLO$-expansion}\/ of $H$.

Let now $M$ be a $\d$-maximal Hardy field extension of~$H$ and 
$H^{\operatorname{da}}$ the $\d$-closure of~$H$ in~$M$, so $H(\R)\subseteq H^{\operatorname{da}}\subseteq M$. 
From Corollary~\ref{cor:nlc, 1} we  obtain: 

\begin{cor}
If $H$ is $\upo$-free, then $H^{\operatorname{da}}$ is a Newton-Liouville closure of~$H(\R)$.
\end{cor}

\noindent
Next, let $\boldsymbol H(\R)$, $\boldsymbol H^{\operatorname{da}}$, $\boldsymbol M$ be the canonical $\HLO$-expansions of the Hardy fields $H(\R)$, $H^{\operatorname{da}}$, $M$, respectively,
so $\boldsymbol H(\R)\subseteq\boldsymbol H^{\operatorname{da}}\subseteq\boldsymbol M$.      Corollary~\ref{cor:nlc, 2} then yields:

\begin{cor}
$\boldsymbol H^{\operatorname{da}}$ is a Newton-Liouville closure of $\boldsymbol H(\R)$.
\end{cor}

\noindent
We finish with an example   justifying the remark after Theorem~B:

\begin{example}
Let $M$ be a maximal Hardy field, and let 
$\mathbf H$ and
$\mathbf F$ be Newton-Liouville closures  of the canonical $\HLO$-expansions of the Hardy fields $\Q$ and $\Q(\ex)$, respectively. We embed  $\mathbf H$ and $\mathbf F$ in $\mathbf M$ in such a way that upon identifying
$\mathbf H$ and $\mathbf F$ with their images in $\mathbf M$ we have
 $H\subseteq F\subseteq M$ for the corresponding underlying differential fields. Now $H\cap F^\dagger=H\cap F= H=H^\dagger$, since~$H$,~$F$ are Liouville closed.
The constant fields of $H$, $F$ are the real closures in~$\R$ of~$\Q$, $\Q(\ex)$, respectively (see~[ADH, proof of  16.4.9]), so $C_H\neq C_F$.  
Corollary~\ref{cor:RosenlichtOrder1, H-fields} yields a~$y\in F\setminus H$ that is $\d$-algebraic over $H$ with
$C_{E}=C_H$ for $E:=H\langle y\rangle$. If $g\in E\cap\exp(H)$,  then~${g^\dagger\in H=H^\dagger}$, so~$g=ch$ where $c\in C_E^\times$ and~$h\in H^\times$,
and hence $g\in H$ since~$C_E=C_H$.
Thus~${E\cap \exp(H) \subseteq H}$.   
\end{example}

\printendnotes

\newlength\templinewidth
\setlength{\templinewidth}{\textwidth}
\addtolength{\templinewidth}{-2.25em}

\patchcmd{\thebibliography}{\list}{\printremarkbeforebib\list}{}{}

\let\oldaddcontentsline\addcontentsline
\renewcommand{\addcontentsline}[3]{\oldaddcontentsline{toc}{section}{References}}

\def\printremarkbeforebib{\bigskip\hskip1em The citation [ADH] refers to our book \\

\hskip1em\parbox{\templinewidth}{
M. Aschenbrenner, L. van den Dries, J. van der Hoeven,
\textit{Asymptotic Differential Algebra and Model Theory of Transseries,} Annals of Mathematics Studies, vol.~195, Princeton University Press, Princeton, NJ, 2017.
}

\bigskip

} 

\bibliographystyle{amsplain}

\end{document}